\documentclass[11pt]{amsart}
\usepackage[british]{babel}

\usepackage{subfiles}

\usepackage{tabu}
\usepackage[margin=1in]{geometry}
\usepackage[dvipsnames]{xcolor}   		             		
\usepackage{graphicx}			
\usepackage{amssymb}
\usepackage{mathrsfs}
\usepackage{amsthm}
\usepackage{amsmath}
\usepackage{stmaryrd}
\usepackage{tikz}
\usepackage{tikz-cd}
\usepackage{accents}
\usepackage{upgreek}
\usepackage{enumerate}
\usepackage{bm}
\usepackage{mathtools}
\usepackage[all]{xy}
\usepackage{caption}

\usepackage{url}
\usepackage{float}
\usepackage{todonotes}
\usepackage{bbm}
\usepackage[shortlabels]{enumitem}

\usepackage[full]{textcomp}
\usepackage[cal=cm]{mathalfa}
\usepackage{xparse}
\usepackage{comment}
\usepackage{cite}
\usepackage{soul}
\usepackage{xfrac}
\usepackage[normalem]{ulem}

\tikzset{
  commutative diagrams/.cd, 
  arrow style=tikz, 
  diagrams={>=stealth}
}

\theoremstyle{definition}

\makeatletter           \def\@tocline#1#2#3#4#5#6#7{\relax
      \ifnum #1>\c@tocdepth 
      \else
        \par \addpenalty\@secpenalty\addvspace{#2}%
        \begingroup \hyphenpenalty\@M
        \@ifempty{#4}{%
          \@tempdima\csname r@tocindent\number#1\endcsname\relax
        }{%
          \@tempdima#4\relax
        }%
        \parindent\z@ \leftskip#3\relax \advance\leftskip\@tempdima\relax
        \rightskip\@pnumwidth plus4em \parfillskip-\@pnumwidth
        #5\leavevmode\hskip-\@tempdima
          \ifcase #1
           \or\or \hskip 1em \or \hskip 2em \else \hskip 3em \fi%
          #6\nobreak\relax
        \dotfill\hbox to\@pnumwidth{\@tocpagenum{#7}}\par
        \nobreak
        \endgroup
      \fi}
\makeatother

\usetikzlibrary{calc}
\usetikzlibrary{fadings}
\usetikzlibrary{decorations.pathmorphing}
\usetikzlibrary{decorations.pathreplacing}

\newcounter{marginnote}
\DeclareMathAlphabet{\mathpzc}{OT1}{pzc}{m}{it}

\usepackage[backref]{hyperref}
\hypersetup{
  colorlinks   = true,          
  urlcolor     = blue,          
  linkcolor    = purple,          
  citecolor   = blue             
}
\usepackage{cleveref}

\newcounter{mainresults}
\theoremstyle{definition}
\newtheorem{theorem}{Theorem}[section]

\newtheorem{corollary}[theorem]{Corollary}
\newtheorem{lemma}[theorem]{Lemma}
\newtheorem{proposition}[theorem]{Proposition}
\newtheorem{remark}[theorem]{Remark}
\newtheorem{assumptions}[theorem]{Assumption}
\newtheorem*{runningexample*}{Running example}

\newtheorem*{aside*}{Aside}

\newtheorem{definition}[theorem]{Definition}
\newtheorem{example}[theorem]{Example}

\newtheorem{notation}[theorem]{Notation}

\newtheorem{proposition-definition}[theorem]{Proposition-Definition}

\newtheorem{maintheorem}[mainresults]{Theorem}

\DeclareMathOperator{\ev}{ev}

\DeclareMathOperator{\Hom}{Hom}

\DeclareMathOperator{\Bl}{Bl}
\DeclareMathOperator{\Pic}{Pic}
\DeclareMathOperator{\frPic}{\mathfrak{Pic}}

\newcommand{\bcd}{\begin{center}\begin{tikzcd}}
\newcommand{\ecd}{\end{tikzcd}\end{center}}

\newcommand{\cC}{\mathcal{C}}
\newcommand{\cD}{\mathcal{D}}
\newcommand{\cE}{\mathcal{E}}
\newcommand{\cF}{\mathcal{F}}

\newcommand{\cL}{\mathcal{L}}
\newcommand{\cO}{\mathcal{O}}
\newcommand{\cQ}{\mathcal{Q}}

\newcommand{\cM}{\mathcal{M}}

\newcommand{\cX}{\mathcal{X}}

\newcommand{\cW}{\mathcal{W}}

\newcommand{\sA}{\mathscr{A}}

\newcommand{\CC}{\mathbb{C}}
\newcommand{\EE}{\mathbb{E}}

\newcommand{\LL}{\mathbb{L}}
\newcommand{\NN}{\mathbb{N}}
\newcommand{\OO}{\mathcal{O}}
\newcommand{\PP}{\mathbb{P}}
\newcommand{\QQ}{\mathbb{Q}}
\newcommand{\RR}{\mathbb{R}}
\newcommand{\TT}{\mathbb{T}}
\newcommand{\ZZ}{\mathbb{Z}}
\newcommand{\AAA}{\mathbb{A}}
\newcommand{\GG}{\mathbb{G}}
\newcommand{\agm}{[\AAA^1/\GG_{m}]}

\newcommand{\vir}{\text{\rm vir}}

\newcommand{\frM}{\mathfrak{M}}

\newcommand{\pt}{\text{pt}}

\newcommand{\Spec}{\operatorname{Spec}}

\newcommand{\GIT}{/\!\!/}

\NewDocumentCommand{\compatibilitydatum}{m m m m m m O{} O{} O{}}{
\begin{equation*} \begin{tikzcd}[ampersand replacement=\&]
  \: \arrow{r} \& {#1} \arrow{r} \arrow{d}{#7} \& {#2} \arrow{r} \arrow{d}{#8} \& {#3} \arrow{r}{[1]} \arrow{d}{#9} \& \: \\
  \: \arrow{r} \& {#4} \arrow{r} \& {#5} \arrow{r} \& {#6} \arrow{r} \& \:
\end{tikzcd} \end{equation*}}

\NewDocumentCommand{\commutingsquare}{m m m m o O{} O{} O{} O{}}{
\begin{equation}\begin{tikzcd}[ampersand replacement=\&] \label{#5}
  #1 \arrow{r}{#6} \arrow{d}{#7} \& #2 \arrow{d}{#8} \\
  #3 \arrow{r}{#9} \& #4
\end{tikzcd}\IfValueTF{#5}{\label{#5}}{} \end{equation}}

\NewDocumentCommand{\Cartesiansquare}{m m m m O{} O{} O{} O{}}{
\begin{equation*}\begin{tikzcd}[ampersand replacement=\&]
  #1 \arrow{r}{#5} \arrow{d}{#6} \arrow[dr, phantom, "\square"] \& #2 \arrow{d}{#7} \\
  #3 \arrow{r}{#8} \& #4
\end{tikzcd} \end{equation*}}

\NewDocumentCommand{\Cartesiansquarelabel}{m m m m m O{} O{} O{} O{}}{
\begin{tikzcd}[ampersand replacement=\&]
  #1 \arrow{r}{#6} \arrow{d}{#7} \arrow[dr, phantom, "\square"] \& #2 \arrow{d}{#8} \\
  #3 \arrow{r}{#9} \& #4
\end{tikzcd}\IfValueTF{#5}{\label{#5}}{}
}

\NewDocumentCommand{\triangleofspaces}{m m m O{} O{} O{}}{
\begin{tikzcd} [ampersand replacement=\&]
#1 \arrow{r}{#4} \arrow[bend right]{rr}{#5} \& #2 \arrow{r}{#6} \& #3
\end{tikzcd}}

\begin{document}

\title{The local/logarithmic Correspondence and the Degeneration Formula for Quasimaps}

\author[Cobos Rabano]{Alberto Cobos Rabano}
\address{School of Mathematics and Statistics, Hicks Building, Hounsfield Rd,
Sheffield S3 7RH,UK}
\email{acobosrabano1@sheffield.ac.uk}
\author[Manolache]{Cristina Manolache}
\address{School of Mathematics and Statistics, Hicks Building, Hounsfield Rd,  Sheffield S3 7RH,UK}
\email{c.manolache@sheffield.ac.uk}
\author[Shafi]{Qaasim Shafi}
\address{School of Mathematics\\
	Watson Building \\
	University of Birmingham\\
        Edgbaston
	B15 2TT\\
	UK}
\email{m.q.shafi@bham.ac.uk}

\begin{abstract}
    We study the relationship between the enumerative geometry of rational curves in local geometries and various versions of maximal contact logarithmic curve counts. Our approach is via quasimap theory, and we show versions of the \cite{vangarrel2019local} local/logarithmic correspondence for quasimaps, and in particular for normal crossings settings, where the Gromov-Witten theoretic formulation of the correspondence fails. The results suggest a link between different formulations of relative Gromov-Witten theory for simple normal crossings divisors via the mirror map. The main results follow from a rank reduction strategy, together with a new degeneration formula for quasimaps.
\end{abstract}

\maketitle
\tableofcontents

\newpage

\section{Introduction}

Let $X$ be a smooth projective variety and $D = D_1 + \dots + D_r \subset X$ be a simple normal crossings divisor whose irreducible components are nef. The \emph{local/logarithmic correspondence} connects two enumerative theories associated to $(X,D)$: rational curves in $X$ with maximal tangency to $D$, and rational curves in the total space of $\cO_{X}(-D)$. This correspondence was conjectured by Takahashi \cite{takahashi2001log} and proved by Gathmann \cite{gathmann2003relative}, for the Gromov--Witten theory of $X= \PP^2$ and $D$ a smooth cubic, and then generalised in \cite{vangarrel2019local} to arbitrary $X$ and $D$ any \emph{smooth}, nef divisor.

\subsection{Results}

We prove a quasimap local/logarithmic correspondence, for simple normal crossings divisors, under a mild positivity assumption.  Let $X = W \GIT G$ be a GIT quotient and let $D = D_1 + \dots + D_r \subset X$ be a simple normal crossings divisor satisfying Assumptions \ref{assumptions:absolute_quasimaps} and \ref{assumptions:log_quasimaps}.

This includes toric varieties and partial flag varieties, with any simple normal crossings divisor, as well as complete intersections in these spaces with divisors pulled back from the ambient space. There are two different directions in which to generalise the local/logarithmic correspondence for smooth divisors. When tangency is imposed with respect to all divisor components simultaneously we prove a local/logarithmic correspondence for arbitrary $r$.

\begin{maintheorem}[\Cref{qloglocalcorner}]\label{Theorem : local log corner}
Suppose $D_1 \cap \dots \cap D_r \neq \emptyset$ and the components of $D$ are very ample. Then we have the following equality of virtual classes
\begin{equation*}
     [\cQ_{0,(\underline{d},0)}^{\log}(X|D,\beta)]^{\vir} = \prod_{i=1}^r(-1)^{d_i+1} \cdot \ev_1^* D_i \cap [\cQ_{0,2}(\oplus_{i=1}^r\cO_{X}(-D_i),\beta)]^{\vir},
\end{equation*}
where on the left hand side maximal tangency is imposed to $D_1 \cap \dots \cap D_r$ at a single point.
\end{maintheorem}

For $r=1$, in Gromov--Witten theory, this is the strong form of the smooth local/logarithmic correspondence (see ~\cite[Conjecture 18]{tseng2023mirror}). The original form is obtained by using the divisor equation and is the main theorem of \cite{vangarrel2019local}. The authors of \cite{vangarrel2019local} conjectured a \emph{different} generalisation for $D$ an s.n.c. divisor, where maximal tangency is imposed to the components of the divisor at distinct points. Figure \ref{fig:cornervsvGGR} exhibits the difference between the generalisations. Although there is some numerical evidence for either generalisation \cite{bousseau2024stable,bousseau2022log}, the authors of \cite{nabijou2022gromov} showed that even if $D = D_1 + D_2$, both generalisations can fail in Gromov--Witten theory. In addition to Theorem \ref{Theorem : local log corner}, we show that the strong form version of the conjecture of \cite{vangarrel2019local} is true for quasimaps for $r=2$.

\begin{maintheorem}[\Cref{snclocallog}]\label{Theorem : local log rank 2}
Suppose that $D=D_1 + D_2$ has two components, which are very ample.  
Then we have the following equality of virtual classes 
$$[\cQ_{0,(d_1,d_2)}^{\log}(X|D_1 + D_2,\beta)]^{\vir} = (-1)^{d_1+d_2} \cdot \prod_{i=1}^2 \ev_i^*D_i \cap [\cQ_{0,2}(\oplus_{i=1}^2\cO_{X}(-D_i),\beta)]^{\vir}.$$
\end{maintheorem}

In fact,~\cite[Conjecture 18]{tseng2023mirror} gives a much more general local/logarithmic conjecture for simple normal crossings divisors. The situations in Theorems \ref{Theorem : local log corner} and \ref{Theorem : local log rank 2} are the two extremes.

In order to prove \Cref{Theorem : local log corner} (as well as \Cref{Theorem : local log rank 2}), we reduce to the case where $r=1$ in 
\Cref{thm:cartesian_log_qmap_snc,thm:cartesian_log_qmap_snc_r}
by showing that, for products of projective spaces relative toric boundary divisors, the simple normal crossings logarithmic quasimap spaces can be obtained as fibre products of logarithmic quasimap spaces for smooth divisors. Then, following the strategy of \cite{vangarrel2019local}, we prove \Cref{Theorem : local log corner} in the case where $r=1$ using a degeneration formula for quasimaps which we develop. 

Let $W \rightarrow \AAA^1$ be a toric double point degeneration. An example is the deformation to the normal cone of a smooth projective toric variety along a smooth toric divisor. Let $X_0$ denote the central fibre, which is a union of two toric varieties $X_1$ and $X_2$ glued along a common toric divisor $D$.

\begin{maintheorem}[Degeneration formula for quasimaps, \Cref{Th: degeneration}]\label{Theorem : Degeneration}
    \begin{equation*}
        [\cQ_{g,n}^{\log}(X_0,\beta)]^{\vir} = \sum_{\widetilde{\Gamma}} \frac{m_{\widetilde{\Gamma}}}{|\mathrm{Aut}(\widetilde{\Gamma})|} j_{\widetilde{\Gamma},*} c_{\widetilde{\Gamma}}^* \, \Delta^! \left(\prod_{V \in V(\widetilde{\Gamma})} [\cQ^{\log}_{g_V,\alpha_V}(X_{r(V)}|D,\beta_V)]^{\vir}\right)
    \end{equation*}
\end{maintheorem} 

Both the proof and application of the degeneration formula for quasimaps entail additional subtleties, when compared with stable maps. The main one is based on the fact that quasimap spaces do not embed along closed embeddings, see \cite[Remark 2.3.3]{battistella2021relative}. As a consequence, if $V$ is a toric subvariety of a toric variety $X$, then quasimaps to $X$ which factor through $V$ contain more information than quasimaps to $V$ with its toric presentation, as we illustrate in \Cref{example ghost 1}. We show in \Cref{ex:P1_relative_infinity} that this phenomenon is relevant when studying toric double point degenerations, and it forces us to consider non-toric presentations for the irreducible components of the special fibre. As a result, the output of the degeneration formula, \Cref{Th: degeneration}, may contain curve classes which live on the degeneration family, rather than on the irreducible components of the special fibre. We call such curve classes \emph{ghost} classes (see \Cref{def:ghost_class}). For the application to the local/logarithmic correspondence, these contributions are ruled out in \Cref{no ghosts}. To prove \Cref{no ghosts} we show that these contributions have a trivial excess line bundle.

\subsection{Motivation}

Relative or logarithmic Gromov–Witten theory \cite{chen2014stable,gross2013logarithmic,abramovich2014stable,Li_relative_morphisms} has significantly shaped the landscape of enumerative geometry. The resulting invariants (and their generalisations) have proved important in mirror symmetry constructions \cite{gross2022canonical}, degeneration formulas \cite{ranganathan2022logarithmic,gross2023remarks,kim2021degeneration,Li_degeneration_formula, abramovich2021punctured}, tropical correspondence theorems such as \cite{mandel2020descendant,gross2018intersection,graefnitz2022tropical,nishinou2006toric} as well as understanding the tautological ring of the moduli space of curves \cite{graber2005relative,ranganathan2023logarithmic,molcho2023case, holmes2022logarithmic}. Consequently, computing logarithmic Gromov--Witten invariants is an important problem in enumerative algebraic geometry. 

An interesting feature of the local/logarithmic correspondence is that the two sides can be understood by qualitatively different techniques, especially when the logarithmic divisor is genuinely simple normal crossings. Thus, conceptually moving to one side and then using its features can be profitable. Although there have been many numerical generalisations and variations of this correspondence \cite{bousseau2021holomorphic,vangarrel2024gromovwitten, bousseau2021stable,bousseau2024stable,bousseau2022log}, the original cycle level correspondence holds in the case where the divisor is smooth. On the other hand, the major difficulty in computing logarithmic Gromov--Witten invariants, at least in genus zero, lies in the case where the divisor is genuinely simple normal crossings.

\subsection{The role of quasimaps}
There have been conjectures to generalise the local/logarithmic correspondence to the case where the divisor is simple normal crossings, but as we have remarked these conjectures are false in general, in Gromov--Witten theory \cite{nabijou2022gromov}. Moreover, the authors noticed ~\cite[Remark 5.4]{nabijou2022gromov} that one point of view for the underlying reason for the failure of these conjectures was the presence of excess components in the moduli space containing \emph{rational tails}. Quasimap theory \cite{marian2011moduli,ciocan2010moduli,ciocan2014stable} provides an alternative framework for counting curves which excludes rational tails, and so it is conceivable that these are no longer counterexamples in the quasimap setting. 

The origins of quasimap theory are intertwined with its connections to mirror symmetry \cite{ciocan2014wall} and the main applications of the theory have been wall-crossing formulas which relate Gromov--Witten invariants to quasimap invariants \cite{ciocan2014wall,ciocan2017higher,zhou2022quasimap}. These wall-crossing formulas reduce computations in Gromov--Witten theory to quasimap theory, which are typically easier. Indeed, almost all results that allow us to compute genus-zero
Gromov–Witten invariants \cite{givental1998mirror} rely at heart on the wall-crossing formula. Recently, the third author has developed a theory of logarithmic quasimaps \cite{shafi2024logarithmic} with the expectation of proving \emph{logarithmic} wall-crossing formulas, first proposed in  \cite{battistella2021relative}. A consequence of \Cref{Theorem : local log corner} and \Cref{Theorem : local log rank 2}, which confirm the predictions of \cite{nabijou2022gromov}, is that the techniques of ordinary Gromov--Witten/quasimap theory can be imported over to compute logarithmic quasimap invariants in the case where the divisor is genuinely simple normal crossings. Combining this with a conjectural wall-crossing should be a powerful method for computing logarithmic Gromov--Witten invariants.

\subsection{Quasimaps and the degeneration formula}

The degeneration formula for double point degenerations \cite{Li_degeneration_formula, chen2014degeneration,kim2021degeneration,abramovich2016orbifold} has been a well used tool in Gromov--Witten theory since its first introduction to the subject \cite{maulik2006topological,okounkov2009gromov,pandharipande2017gromovpairs}. For all the same reasons the degeneration formula for quasimaps, \Cref{Theorem : Degeneration}, will provide similar computational capabilities. In this paper we use the degeneration formula to prove the local/logarithmic correspondence for smooth pairs. There is a related direction to this, which is the prospect of investigating holomorphic anomaly equations in the relative setting, for which quasimaps are better suited. 

In the absolute setting, for Calabi-Yau threefolds, the wall-crossing formula coincides with the mirror map and hence quasimap invariants are exactly equal to the B-model invariants of the mirror. This link has been exploited to give a direct geometric proof, for local $\PP^2$, of the holomorphic anomaly equation \cite{lho2018stable}, which gives profound structure to the Gromov--Witten theory of a Calabi--Yau, coming from the B-model. Using the double point degeneration formula, the authors of \cite{bousseau2021holomorphic} prove a holomorphic anomaly equation for the logarithmic Gromov–Witten theory of $\PP^2$
relative an elliptic curve. This suggests that more general holomorphic anomaly equations in relative geometries could be pursued using the degeneration formula for quasimaps, after generalising this beyond the toric setting.

\subsection{Future directions}

\subsubsection{Local/logarithmic correspondence in higher genus}

There are two ways in which one can generalise the original smooth local/logarithmic correspondence. Either to increase the rank of the divisor, as we have addressed in the present paper, or to generalise from rational curves to higher genus curves, where on the logarithmic side one adds a $\lambda_g$ class for the dimensions to match. In Gromov--Witten theory both generalisations fail on the nose, for the higher genus situation there are correction terms described in \cite{bousseau2021holomorphic}. One can speculate about whether, analogous to the situation for higher rank divisors, there are instances where the correction terms vanish for the quasimap local/logarithmic correspondence in higher genus.

\subsubsection{Holomorphic anomaly equations}

In the Gromov--Witten situation, the relationship between local and logarithmic invariants in higher genus is used to prove a holomorphic anomaly equation for $\PP^2$ relative to an elliptic curve. Since quasimaps are better suited to understanding holomorphic anomaly equations, a natural is to pursue this for other geometries.

In order to accomplish this, the degeneration formula would need to be generalised beyond the toric setting. This should be possible however. The main difficulty in quasimap theory is the restriction of the theory to spaces admitting a certain GIT presentation. Even the deformation to the normal cone of a toric variety along a smooth, non-toric divisor is no longer toric. On the other hand, one can rewrite the total space, as well as more general blow-ups, as a GIT quotient, see, for example,  ~\cite[Proposition 6.2]{coates2022abelian}.

\subsubsection{Orbifold/logarithmic correspondence}

The local/logarithmic correspondence concerns invariants with maximal contact to the divisor which makes it fertile testing ground, but many of the ideas pursued here should extend to a more general context. An alternative approach to tangency involves replacing components of divisor with roots \cite{cadman2007using}, and using the framework of orbifold Gromov--Witten theory to produce invariants \cite{tseng2023gromov}. There has been effort to compare the logarithmic theories with the resulting orbifold theory \cite{abramovich2017relative,battistella2023gromovwitten,tseng2020higher}, since the latter is significantly easier to compute in practice. However, it computes a different set of invariants in general, which do not immediately interact with degeneration formulas and lacks some of the conceptual advantages of the logarithmic theory. Furthermore, no known structure organises the difference between the theories. 

Our work here sheds some light on this difference. One can view the local theory defined by an s.n.c. divisor as a sector of this orbifold by the results of \cite{battistella2023local}, and so our results suggest that the difference between the computable orbifold theory and the conceptually better logarithmic theory, might be controlled in part by mirror maps. There is a related proposal in \cite{you2022relative}. 

Taking these issues together, it seems reasonable that orbifold and logarithmic Gromov--Witten theory for pairs should be compared via the quasimap theory, where at least in some sectors such as the one we study here, the invariants coincide. The general picture appears to be more complicated, but the significantly simplified combinatorics of quasimaps provides a clear path to computations.

\subsection{Acknowledgements}
The suggestion of a quasimap local/logarithmic correspondence arose from discussions with Michel van Garrel, and we would like to wholeheartedly thank him here. We are very grateful to Yannik Sch\"uler for his explanations on a localisation computation in \cite{Bryan_Pandharipande}, which helped us prove \Cref{Prop : qlocal P1 comp}. We thank Leo Herr and Navid Nabijou for useful discussions and Dhruv Ranganathan for discussions and very helpful comments on the draft. We thank Navid Nabijou, Helge Ruddat and Fenglong You for their comments on a previous version.

C.M. was partially supported by a Royal Society Dorothy Hodgkin Fellowship. Q.S. is supported by UKRI Future Leaders Fellowship through
grant number MR/T01783X/1. 

\section{Absolute quasimaps}

Here we recall the definition of absolute GIT quasimaps from \cite{ciocan2014stable}.

\begin{assumptions}\label{assumptions:absolute_quasimaps}
Let $Z= \Spec A$ be an affine algebraic variety with an action by a reductive algebraic group $G$. Let $\theta$ be a character inducing a linearisation for the action. We insist that
\begin{itemize}
	\item $Z^s = Z^{ss}$
	\item $Z^s$ is nonsingular
	\item $G$ acts freely on $Z^s$
	\item $Z$ has only l.c.i singularities
\end{itemize}
\end{assumptions}

Then quasimaps to $Z \GIT G$ are defined as follows. 

\begin{definition}\label{absquasimapdef}
Fix non-negative integers $g,n$ and $\beta \in \Hom(\Pic^G Z, \ZZ)$. An $n$-marked \textit{stable quasimap} of genus $g$ and degree $\beta$ to $Z \GIT G$ is 
	\begin{equation*}
	\left((C, p_1, \dots, p_n), u \right)
	\end{equation*}	
where
	\begin{enumerate}
		\item $(C, p_1, \dots, p_n)$ is an $n$-marked prestable curve of genus $g$.
		\item $u: C \rightarrow [Z/G]$ of class $\beta$  (i.e. $\beta(L) = \deg_C u^*L$).
	\end{enumerate}
	which satisfy 
	\begin{enumerate}
		\item (Non-degeneracy) there is a finite (possibly empty) set $B \subset C$, distinct from the nodes and markings, such that $\forall c \in C \setminus B$ we have $u(c) \in Z^s$. 
		\item (Stability) $\omega_C(p_1 + \cdots + p_n) \otimes L_{\theta}^{\epsilon}$ is ample $\forall \epsilon \in \QQ_{>0}$, where $L_{\theta}$ comes from the linearisation.
	\end{enumerate}
\end{definition}

\begin{theorem}[\hspace{-0.01em}\cite{ciocan2014stable}]
    The moduli stack of stable quasimaps $\cQ_{g,n}(X,\beta)$ is a Deligne--Mumford stack. If $X$ is proper then so is $\cQ_{g,n}(X,\beta)$. Moreover, it admits a perfect obstruction over $\frM_{g,n}$ leading to a virtual fundamental class $[\cQ_{g,n}(X,\beta)]^{\vir}$.
\end{theorem}

In an earlier iteration of the theory in \cite{ciocan2010moduli}, \textit{toric} quasimaps are defined for toric varieties using the description of the functor of points of a smooth toric variety in terms of line bundle-section pairs according to rays of the fan, proved in \cite{cox1995functor}. In the general context of \Cref{absquasimapdef}, toric quasimaps are recovered by taking the standard GIT presentation coming from toric geometry. 

We will prove a degeneration formula in a toric context, which means the double point degeneration will be a \emph{toric morphism}. However, we will still require the more general definition of quasimaps, \Cref{absquasimapdef}, even in this context, because the process of gluing a quasimap is more subtle than in the case of maps (see the discussion and examples in \Cref{subsec:quasimaps_special_fibre}). Specifically in order to glue together quasimaps to the components of the central fibre we require the GIT presentation induced by the embedding in the total space, rather than the ordinary toric presentation.

\begin{definition}
    A \emph{toric double point degeneration} $W \rightarrow \AAA^1$ is a double point degeneration with $W$ a toric variety and $W \rightarrow \AAA^1$ a toric morphism. By double point degeneration we mean that $W \rightarrow \AAA^1$ is a flat, proper morphism, with $W$ smooth. The general fibre is smooth and the fibre over $0 \in \AAA^1$ is a reduced union of two irreducible components glued along a common smooth divisor.
\end{definition}

\begin{notation}\label{notation:degenerations}
    We denote the general fibre of a toric double point degeneration $W\to \AAA^1$ by $X$. The central fibre $X_0$ is a union of the two components $X_1$ and $X_2$ glued along a common divisor smooth $D$. Necessarily the fan of $W$ has two distinguished rays which map surjectively to the ray in the fan of $\AAA^1$. These rays correspond to the divisors $X_1$ and $X_2$ of $W$, and will typically be denoted by $\rho_1$ and $\rho_2$. All the other rays are contained in a codimension 1 sublattice, as in \Cref{fig:fan_double_point_degeneration}. We will denote the toric GIT presentation of $W$ by $\AAA^M \GIT \GG_m^s$ and the ambient stack $[\AAA^M/ \GG_m^s]$ by $\cW$.
\end{notation}

\begin{figure}
    \centering
    \includegraphics[width=5cm]{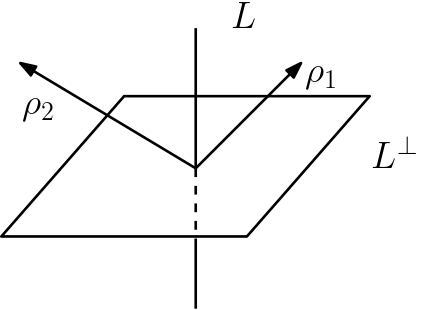}
    \caption{A double point degeneration has two special rays $\rho_1,\rho_2$, which project to the generator of the fan of $\AAA^1$ inside a line $L$. All other rays are contained in the hyperplane $L^\perp$.}
    \label{fig:fan_double_point_degeneration}
\end{figure}

\begin{remark}
    Just as in Gromov--Witten theory one can consider quasimaps over a base. Since this is not addressed in the literature we just make the observation that for a double point degeneration $W \rightarrow \AAA^1$ we have that $\cQ_{g,n}(W/\AAA^1,\beta) = \cQ_{g,n}(W,\beta)$, because any quasimap to $W$ necessarily lives in a fibre over $\AAA^1$. 
\end{remark}

\subsection{Quasimaps to a double point degeneration} \label{subsec:quasimaps_special_fibre}

The degeneration formula involves the study of quasimaps to a double point degeneration $W\to \AAA^1$, which we assume to be toric. Let $\AAA^M \GIT \GG_m^s$ be the toric GIT presentation of $W$. We define quasimaps to $X_0$ via the following Cartesian diagram.

\[
    \begin{tikzcd}
     \cQ_{g,n}(X_0,\beta) \arrow{r}\arrow{d} & \cQ_{g,n}(W,\beta)\arrow{d}\\
    0 \arrow{r} & \AAA^1
    \end{tikzcd}
\]

The degeneration formula expresses will express the virtual class of $[\cQ_{g,n}(X_0,\beta)]^{\vir}$, defined via Gysin pullback in the above diagram, with the virtual classes of logarithmic quasimap spaces of $X_1$ and $X_2$ relative to $D$. For stable maps, one has that a map to $X_0$ is the same as a collection of maps to $X_i$ gluing at the nodes. The analogue statement for quasimaps is true only if we take the GIT presentation $X_i \simeq \AAA^{M-1}\GIT \GG_m^s$ induced by $W$, which in general does not agree with the standard toric presentation. 
\Cref{ex:P1_relative_infinity} provides a concrete example of this phenomenon.

\begin{figure}
    \includegraphics[width=5cm]{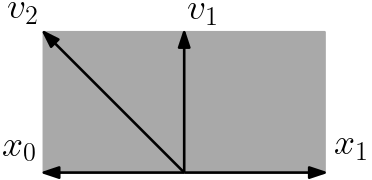}
		\caption{The fan of $W=\Bl_{\infty\times 0}(\PP^1\times \AAA^1)$.}
		\label{fig:fan_P1_relative_infinity}
    \end{figure}
    
\begin{example}\label{ex:P1_relative_infinity}
    Consider the deformation to the normal cone of $\PP^1$ relative $\infty$, $W=\Bl_{\infty\times 0}(\PP^1\times\AAA^1)$. 
    The fan of $W$, pictured in \Cref{fig:fan_P1_relative_infinity}, induces the following GIT presentation
    \[
        W = \AAA^4\GIT \GG_m^2, \, \, \mathrm{where} \, \,  (\lambda,\mu)\cdot (x_0,x_1,v_1,v_2) = (\lambda\mu x_0,\lambda x_1,\mu v_1, \mu^{-1}v_2).
    \]
    Setting $v_1=0$ in the presentation of $W$ induces the following presentation of $X_1 = \PP^1$:
    \[
        X_1 = \AAA^3 \GIT \GG_m^2.
    \]
    On the other hand, $\PP^1$ has the following toric presentation
    \[
        \PP^1 = \AAA^2 \GIT \GG_m  \, \, \mathrm{where} \, \, \alpha\cdot (z_0,z_1) = (\alpha z_0, \alpha z_1).
    \] 
    There is a  morphism $f : \left[\AAA^3/\GG_m^2\right] \to \left[\AAA^2/\GG_m\right]$ given by
    \begin{align*}
        (x_0,x_1,v_2) &\mapsto (x_0\cdot v_2, x_1)\\
        (\lambda, \mu)&\mapsto  \lambda
    \end{align*}
    and a morphism $g: \left[\AAA^2/\GG_m\right] \to \left[\AAA^3/\GG_m^2\right]$ given by 
    \begin{align*}
        (z_0,z_1) &\mapsto (1,z_1,0,z_0)\\
        \alpha &\mapsto  (\alpha, \alpha^{-1}).
\end{align*}
Note that $f\circ g = \mathrm{id}$ but $g\circ f\neq \mathrm{id}$ (one must restrict to the stable locus for $f$ and $g$ to be inverse). As a consequence, every quasimap to $\PP^1$ with the toric presentation appears as the restriction of a quasimap to $\PP^1$ with the presentation induced by $W$, but the converse is not true. 
\end{example}

Loosely speaking, \Cref{ex:P1_relative_infinity} shows that a quasimap to $X_1$ with the presentation induced by $W$ requires strictly more information than with the standard presentation. For instance, the degree of a quasimap to $X_i$ in the first case lies in $H_2(W,\ZZ)$, which in general may differ from $H_2(X_i,\ZZ)$. 

\begin{example}\label{example ghost 1}
    For any $k \geq 0$, consider the quasimap 
    \begin{align*}
        q_k : \PP^1 &\dashrightarrow \PP^1\times\PP^1\\
        [s\colon t] &\mapsto ([0\colon s^k],[s\colon t])
    \end{align*}
    Note that $x = [0\colon 1]$ is a basepoint of $q_k$ for $k\geq 1$. We can think of $q_k$ as giving a quasimap to $\PP^1$ with the presentation
    \[
    \PP^1 = \AAA^1 \times \AAA^2 \GIT (\GG_m \times \GG_m) \, \, \mathrm{where} \, \, (\lambda, \mu) \cdot (x_1, y_0, y_1) = (\lambda x_1, \mu y_0, \mu y_1)
    \]
    On the other hand, since $q_k$ factors through $[0\colon 1]\times \PP^1$, there is an obvious induced quasimap to $\PP^1$ with the standard presentation given by 
        \begin{align*}
         \PP^1 &\dashrightarrow \PP^1\\
        [s\colon t] &\mapsto [s\colon t].
    \end{align*}
    But this is independent of $k$ and so it is impossible to recover back the original quasimap $q_k$.
\end{example}

\begin{definition}\label{def:ghost_class}
    Let $W$ be a toric variety and $\iota\colon V\hookrightarrow W$ the closure of a torus orbit. A \textit{ghost class} of $V\hookrightarrow W$ is a curve class in $H_2(W,\ZZ)\setminus \iota_\ast (H_2(V,\ZZ))$.
\end{definition}

For example, the quasimap $q_k$ from \Cref{example ghost 1} has degree $(k,1)$, which is a ghost class unless $k=0$. The following example describes ghost classes when the ambient $W$ is the deformation to the normal cone of a hyperplane in projective space, which we will encounter when dealing with the degeneration formula in \Cref{sec:log-local_correspondence}.

\begin{example}\label{deformationPN}
    Let $W(\PP^N,H) = \Bl_{H\times 0} (\PP^N\times\AAA^1)$ be the deformation to the normal cone of a hyperplane $H$ in $\PP^N$, which is a toric double point degeneration with toric GIT presentation 
    \[
        \AAA^{N+3}\GIT \GG_m^2, \, \, \mathrm{where} \, \,  (\lambda,\mu)\cdot (x_0,x_1,\ldots, x_N,v_1,v_2) = (\lambda\mu x_0,\lambda x_1,\ldots, \lambda x_N,\mu v_1, \mu^{-1}v_2).
    \]
    The fan of $W(\PP^N,H)$ is pictured in \Cref{fig:fan_P1_relative_infinity} for $N=1$ and in \Cref{fig:fan_degeneration_normal_cone_P2} for $N=2$. We compute the ghost classes of the inclusion of $X_1\simeq \PP^N$ in $W(\PP^N,H)$. 

    \begin{figure}
        \centering
        \includegraphics[width=7cm]{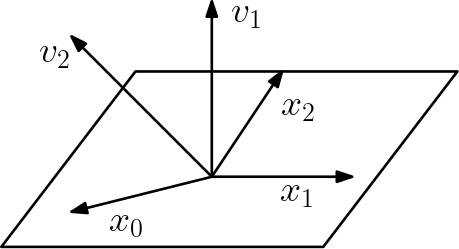}
        \caption{The fan of $W(\PP^2,H) = \Bl_{H\times 0} \PP^2\times \AAA^1$.}
        \label{fig:fan_degeneration_normal_cone_P2}
    \end{figure}

    Let $H_0,\ldots, H_N$ be the toric divisors on $\PP^N$ with $H=H_0$, and let $\tilde{H}_i$ be the strict transform of $H\times \AAA^1$. The toric divisors of $W(\PP^N,H)$ are: 
    \begin{itemize}
        \item $\widetilde{H}_1,\ldots, \widetilde{H}_N$; 
        \item $\widetilde{H}$, the infinity section;
        \item $E$, the exceptional divisor;
        \item $\PP^N$ embedded in the special fiber.
    \end{itemize}
    They satisfy the following relations in $\Pic(W(\PP^N,H))$:
    \begin{itemize}
        \item $[\tilde{H}_1] = \ldots = [\tilde{H}_N] = [\tilde{H}]+[E]$ and
        \item $[E]+[\PP^N]=0$.
    \end{itemize}

    In order to describe ghost classes of $\PP^N$, we need to describe the morphism 
    \begin{equation}\label{eq:morphism_curve_classes}
        H_2(\PP^N,\ZZ) \to H_2(W(\PP^N,H),\ZZ)
    \end{equation}
    induced by the inclusion $\PP^N\hookrightarrow W(\PP^N,H)$ in the special fibre. On the one hand, $H_2(\PP^N,\ZZ) \simeq \ZZ \ell$, with $\ell$ the class of a line. On the other hand,  $H_2(W(\PP^N,H),\ZZ) \simeq \ZZ l \oplus \ZZ f$ with $f$ the class of a fiber $X_2\to H$ and $l$ the class of a line in the general fibre of $W(\PP^N,H)\to\AAA^1$. We fix this basis of $H_2(W(\PP^N,H),\ZZ)$ from now on. Note that a class $(d,e)\in H_2(W(\PP^N,H),\ZZ)$ is effective if and only if $d\geq 0$ and $d+e\geq 0$.

    The transpose $\varphi\colon \Pic(W(\PP^N,H)) \to \Pic(\PP^N)$ of the morphism in \eqref{eq:morphism_curve_classes} is determined by $\varphi([\widetilde{H}_1]) = [H_0]$ and $\varphi([\PP^N]) = -[H_0]$. Combining this with the fact that the basis $\{[\widetilde{H}_1],[\PP^N]\}$ of $\Pic(W(\PP^N,H))$ and $\{l,f\}$ of $H_2(W(\PP^N,H))$ are dual,
    we conclude that
    \begin{align*}
        H_2(\PP^N,\ZZ) &\to H_2(W(\PP^N,H),\ZZ)\\
        \ell &\mapsto (1,-1).
    \end{align*}
    This means that an effective curve class $(d,e)\in H_2(W(\PP^N,H),\ZZ)$ is ghost for $\PP\hookrightarrow W(\PP^N,H)$ if and only if $d\neq -e$. 
\end{example}

\section{Logarithmic quasimaps}
In this section we will introduce logarithmic quasimaps to a toric double point degeneration analogous to the logarithmic quasimaps of \cite{shafi2024logarithmic}. We show that this moduli space admits a virtual fundamental class using the construction of \cite{behrend1997intrinsic}. First, we recall the main theorem of \cite{shafi2024logarithmic} in the specific case where the divisor is smooth. 

\begin{assumptions}\label{assumptions:log_quasimaps}
 Let $X = Z\GIT G$ be a GIT quotient and $D \subset X$ a simple normal crossings divisor. In addition to the requirements of the previous section we impose that $Z \GIT G$ is a subvariety of a vector space quotient $V \GIT G$ from which $D$ is pulled back. 
\end{assumptions}

Suppose that $D$ is a smooth divisor and let $\sA = [\AAA^1 /\GG_m]$.
\begin{theorem}[\hspace{-0.01em}{\cite[Theorems 0.1 $\&$ 0.2]{shafi2024logarithmic}}]\label{quasimaps obs} 
 Let, $g,n \in \NN$, let $\beta$ a quasimap degree and let $\alpha \in \NN^{n}$ such that $\sum_{i=1}^n \alpha_i = D \cdot \beta$. The moduli space $\cQ^{\log}_{g,\alpha}(X|D,\beta)$ parametrising logarithmic quasimaps to $(X,D)$ is a proper Deligne--Mumford stack. If $u$ is the universal map, and $\pi$ is the projection from the universal curve, then the complex $(\mathbf{R}\pi_*u^*\mathbb{T}^{\log}_{\cX})^{\vee}$ is a perfect obstruction theory for the morphism $\cQ^{\log}_{g,\alpha}(X|D,\beta)\to \frM^{\log}_{g,\alpha}(\sA)$.
\end{theorem}

Now we will move to the case of a double point degneration, in the toric setting.

Let $W \rightarrow \AAA^1$ be a toric double point degeneration. Equipping $W$ (resp. $\AAA^1$) with the divisorial logarithmic structure with respect to the central fibre (resp. the origin) makes $W \rightarrow \AAA^1$ into a logarithmically smooth morphism. The Artin fans \cite{abramovich2018birational} of the respective targets are $\sA^2 = [\AAA^2/\GG_m^2]$ and  $\sA = \agm$. Consequently, we have a diagram
\begin{equation}\label{eq:diagram_W_maps_to_curly_A_prime}
    \begin{tikzcd}
W \arrow[r] \arrow[d] & \mathscr{A}^2 \arrow[d] \\
\AAA^1 \arrow[r]        & \mathscr{A}            
\end{tikzcd}
\end{equation}
This diagram is not Cartesian, but $W$ maps to the fibre product $\sA'= \sA^2 \times_{\sA} \AAA^1$, which lives over $\AAA^1$.

As in Notation \ref{notation:degenerations}, $W \subset \cW$, the ambient quotient stack. It follows that $W \rightarrow \sA^2$ factors through $\mathcal{W}$.

\begin{definition}
  Let $C \rightarrow \mathcal{W}$ be a quasimap to $W$ (over $\AAA^1$). Define the \emph{induced map} $C \rightarrow \sA'$ via the composition $$C \rightarrow \mathcal{W} \rightarrow \sA'.$$
\end{definition}

The morphism $\sA' \rightarrow \AAA^1$ is a logarithmic morphism. Write $\frM^{\log}_{g,n}(\sA'/\AAA^1)$ for the moduli space of minimal prestable logarithmic maps to $\sA'/\AAA^1$. 
\begin{definition}\label{logquasimapdef}
Let $g,n$ be non-negative integers, let $\beta$ be an effective curve class on $W$. Define $\cQ^{\log}_{g,n}(W/\AAA^1,\beta)$ as the fibre product of stacks 
\begin{equation*}
\begin{tikzcd}
\cQ^{\log}_{g,n}(W/\AAA^1,\beta) \arrow[d] \arrow[r] & {\frM_{g,n}^{\log}(\mathscr{A}'/\AAA^1)} \arrow[d, "\epsilon_2"] \\
\cQ_{g,n}(W/\AAA^1,\beta) \arrow[r, "\epsilon_1"]  & \frM_{g,n}(\mathscr{A}'/\AAA^1)        
\end{tikzcd}
\end{equation*}
Here $\frM_{g,n}(\sA'/\AAA^1)$ is the stack of maps from $n$-marked, genus $g$ prestable curves to $\sA'$ over $\AAA^1$. The morphism $\epsilon_1$ is given by associating to a stable quasimap the induced map to $\sA'$,
and $\epsilon_2$ is given by forgetting the logarithmic structure.
\end{definition}

\begin{lemma}
   The moduli space $\cQ^{\log}_{g,n}(W/\AAA^1,\beta)$ is a Deligne--Mumford stack and $\cQ^{\log}_{g,n}(W/\AAA^1,\beta) \rightarrow \cQ_{g,n}(W/\AAA^1,\beta)$ is proper.
\end{lemma}
\begin{proof}
    Follows from the fact that $\epsilon_2$ is representable and proper (see ~\cite[2.6-2.8]{shafi2024logarithmic}).
\end{proof}

Next we show that $\cQ^{\log}_{g,n}(W/\AAA^1)$ admits a perfect obstruction theory over $\frM^{\log}_{g,n}(\sA'/\AAA^1)$ following a similar argument to ~\cite[Section 4]{shafi2024logarithmic}. Note that we have a diagram

\begin{equation*}
\begin{tikzcd}
\cC_{\cQ} \arrow[d, "\pi"] \arrow[r] \arrow[rr, "u", bend left] & \cC_{\frM} \arrow[d] \arrow[rr, bend right] & {\mathcal{W}} \arrow[r] & \sA' \arrow[d] \\
\cQ^{\log}_{g,n}(W/\AAA^1,\beta) \arrow[r]                                                   & \frM^{\log}_{g,n}(\sA'/\AAA^1) \arrow[rr]                             &             & \AAA^1        
\end{tikzcd}
\end{equation*}

\begin{lemma}
    There is a morphism in the derived category of $\cQ^{\log}_{g,n}(W/\AAA^1,\beta)$
\begin{equation*}
   \phi: \mathbf{R} \pi_{*}(\mathbf{L}u^* \LL^{\log}_{\mathcal{W}/\AAA^1} \otimes \omega_{\pi}) \rightarrow \LL_{\cQ^{\log}_{g,n}(W/\AAA^1,\beta)/\frM_{g,n}^{\log}(\mathscr{A}'/\AAA^1)}.
\end{equation*}
\end{lemma}
\begin{proof}
   We have that $\LL^{\log}_{\mathcal{W}/\AAA^1} \cong \LL_{\mathcal{W}/\sA'}$ by ~\cite[Corollary 5.25]{olsson2003logarithmic} and so there is a morphism $$\mathbf{L}u^*\LL^{\log}_{\mathcal{W}/\AAA^1} \rightarrow \LL_{\cC_{\cQ/\cC_{\frM}}} \cong \mathbf{L}\pi^* \LL_{\cQ/\frM}.$$
    By tensoring with the dualising sheaf and applying $\mathbf{R}\pi_*$ we get 
    \[ 
        \phi: \mathbf{R} \pi_{*}(\mathbf{L}u^* \LL^{\log}_{\mathcal{W}/\AAA^1} \otimes \omega_{\pi}) \rightarrow \LL_{\cQ/\frM}. \qedhere
    \]
\end{proof}

\begin{proposition}\label{prop : perfectness}
    The morphism $\phi$ defines an obstruction theory and $\mathbf{R} \pi_{*}(\mathbf{L}u^* \LL^{\log}_{\mathcal{W}/\AAA^1} \otimes \omega_{\pi})$ is of perfect amplitude contained in $[-1,0]$.
\end{proposition}

\begin{proof}
    The proof follows the argument of ~\cite[Proposition 4.5]{shafi2024logarithmic}. The only change to the argument is that we are working over the base $\AAA^1$ and so ~\cite[Theorem 5.9]{olsson2003logarithmic} instead gets applied to the following diagram
\begin{equation*}
    \begin{tikzcd}
\cC_T \arrow[rd] \arrow[rdd] \arrow[r] & \cC_{\bar{T}} \arrow[rd, dashed] \arrow[rdd] &                              \\
                                       & \cW \times T \arrow[d] \arrow[r]             & \cW \times \bar{T} \arrow[d] \\
                                       & \AAA^1 \times T \arrow[r]                    & \AAA^1 \times \bar{T}       
\end{tikzcd}
\end{equation*}
Let $\mathbb{T}^{\log}_{\cW/\AAA^1}$ denote the derived dual of $\LL_{\cW/\AAA^1}^{\log}$.

To prove that $\phi$ defines a perfect obstruction theory we use an argument similar to \cite[Proposition 4.7] {shafi2024logarithmic}. More precisely, by \Cref{cor: no neg coh} and generic non-degeneracy on the space of quasimaps, we have that the complex $\mathbf{L}u^*\mathbb{T}^{\log}_{\cW/\AAA^1}$ has cohomology supported in $0$. This shows that the complex $\mathbf{R}\pi_*\mathbf{L}u^*\mathbb{T}^{\log}_{\cW/\AAA^1} \cong (\mathbf{R} \pi_{*}(\mathbf{L}u^* \LL^{\log}_{\mathcal{W}/\AAA^1}\otimes \omega_{\pi}))^{\vee}$ is of perfect amplitude in $[0,1]$.
\end{proof}

\Cref{cor: no neg coh} will follow from \Cref{prop:log_tangent_curlyW_A1} below.

\begin{proposition}\label{prop:log_tangent_curlyW_A1} 
Let $W\to\AAA^1$ be a toric double point degeneration and let $\cW=[\AAA^M/\GG_m^s]$ be the ambient stack. Let $\{D_{\rho_i}\}_{i=1}^M$ denote the toric divisors on $W$ with $X_1 = D_{\rho_1}$ and $X_2 = D_{\rho_2}$.
Then the logarithmic tangent complex of $\cW$ relative to $\AAA^1$ is quasi-isomorphic to the following complex supported in $[-1,1]$ 
\begin{equation}
\TT_{\cW/\AAA^1}^{\log} = 
[
\begin{tikzcd}
\OO_{\cW}^{\oplus s} \arrow{r}{f_1} &\oplus_{\rho\neq \rho_1,\rho_2\in\Sigma_W(1)}\cO_{\cW} (D_{\rho})  \oplus_{i=1}^{2} \cO_{\cW} \arrow{r}{f_2} & \cO_{\cW}
\end{tikzcd}
],
\end{equation}
where $f_2$ maps the generator of each $\OO_{\cW} (D_{\rho})$ to 0 and is the identity on each of the trivial factors. 
\end{proposition}

\begin{proof} 
Recall that we have a diagram
\begin{equation}\label{eq:diagram_log_tangent_W}
    \begin{tikzcd}
\cW \arrow[r,"\epsilon"]&\mathscr{A} '\arrow[r] \arrow{d}{\alpha} & \mathscr{A}^2 \arrow[d] \\
&\AAA^1 \arrow[r]        & \mathscr{A}            
\end{tikzcd}
\end{equation}
The logarithmic tangent complex of $\cW$ is equivalently, the relative tangent complex $\TT_{\cW/\sA'}$. The top row of diagram \eqref{eq:diagram_log_tangent_W} induces an exact triangle
\[
    \TT_{\cW/\sA'}\to \TT_{\cW/\sA^2} \to \mathbf{L}\epsilon^*\TT_{\sA'/\sA^2} \to .
\]
Since $\sA^2\to \sA$ is flat, $\TT_{\sA'/\sA^2} = \alpha^\ast \TT_{\AAA^1/\sA}$ is quasi-isomorphic to $[\cO]$.

On the other hand, we can compute $\TT_{\cW/\sA^2}$ from the exact triangle
\[
    \TT_{\cW/\sA^2}\to \TT_{\cW} \to \TT_{\sA^2}.
\]
The complex $\TT_{\cW}$ is quasi-isomorphic to
\begin{equation}\label{eq:tangent_complex_curly_W}
    [\OO^{\oplus s} \to \oplus_{\rho\in\Sigma_W(1)} \OO(D_{\rho})]
\end{equation}
induced by the weight matrix of the action of $\GG_m^s$ on $\AAA^M$, which can be made explicit as follows. Using that $W$ is a toric double point degeneration with special fibre $D_{\rho_1}+D_{\rho_2}$, we can find a basis of $\Pic(W)$ of the form $\{D_{\rho_1},D_{\rho_{i_2}}, \ldots, D_{\rho_{i_s}}\}$ with $i_j\notin \{1,2\}$. This choice of basis for $\Pic(W)$ induces a choice of generators $\{\lambda_1, \ldots, \lambda_s\}$ of $\Hom(\Pic(W),\ZZ)\otimes_\ZZ \CC^*\simeq \GG_m^s$.  

With this choice of generators, the action has expression
\[
    \lambda_i\cdot (x_1,\ldots, x_M) = (\lambda_i^{a_{i,1}} x_1, \ldots, \lambda_i^{a_{i,M}} x_M)
\]
such that the weight matrix $(a_{i,j})_{1\leq i\leq s, 1\leq j\leq M}$ for the action of $\GG_m^s$ in this basis satisfies that $a_{1,1} = 1$, $a_{1,2} = -1$ and $a_{i,1} = a_{i,2} = 0$ for each $2\leq i\leq s$. Then the morphism in \cref{eq:tangent_complex_curly_W} is given by 
\[
    1_i\mapsto (y_{i,j})_{j=1}^M \text{ with } y_{i,j} =  \partial_{\lambda_i} (\lambda_i^{a_{i,j}} x_j)\mid_{\lambda_i = 1} = \begin{cases}
        0 & \text{if } a_{i,j} = 0\\
        a_{i,j} x_j & \text{if } a_{i,j} \neq 0
    \end{cases}
\]
Using this presentation, we get that $\TT_{\cW/\sA^2}$ is quasi-isomorphic to the complex
\[
    [\cO^{\oplus s} \to \oplus_{\rho\neq\rho_1,\rho_2}\cO(D_{\rho})\oplus\cO^{\oplus 2}]
\]
with morphism
\begin{equation}\label{eq:tangent_complex_curly_W_relative_curly_A_2}
    1_i \mapsto \begin{cases}
        (-y_{1,j})_{j=1}^M \oplus (-1,1)& \text{if } i= 1\\
        (-y_{i,j})_{j=1}^M \oplus (0,0)& \text{if } i\neq 1
    \end{cases}
\end{equation}

The result is obtained by computing the cone of $\TT_{\cW/\sA^2} \to \TT_{\sA'/\sA^2}$.
\end{proof}
\begin{corollary}\label{cor: no neg coh}
   The complex $\mathbf{L}u^*\mathbb{T}^{\log}_{\cW/\AAA^1}$ is a perfect complex in $[-1,0]$ and it has cohomology supported in degree $0$. 
\end{corollary}
\begin{remark}\label{remark: vfc central fibre}
    The moduli space $\cQ^{\log}_{g,n}(W/\AAA^1,\beta)$ lives over $\AAA^1$. We denote the central fibre by $\cQ^{\log}_{g,n}(X_0,\beta)$, where $X_0$ is the central fibre of $W$. Let $i_0 : 0 \hookrightarrow \AAA^1$ be the natural embedding. There is an induced pullback perfect obstruction theory on $\cQ^{\log}_{g,n}(X_0,\beta)$ where the associated virtual class coincides with $i_0^![\cQ^{\log}_{g,n}(W/\AAA^1,\beta)]^{\vir}$, by functoriality of virtual pullbacks ~\cite[Theorem 4.8]{manolache12virtual}. 
\end{remark}

\subsection{The Ciocan-Fontanine--Kim obstruction theory}\label{cfk obs}

In the following we define an analogue of the obstruction theory given in ~\cite[Section 5.3]{ciocan2010moduli} for the situation of \Cref{quasimaps obs} as well as of \Cref{prop : perfectness}. 

\begin{definition}\label{def:Pic_curly_A}
   Let $g,n, d$ be non-negative integers and let $\alpha\in\mathbb{N}^n$. Let $\frPic_{g,\alpha,d}^{\log}(\sA)$ be defined to be the fibre product
\begin{equation*}
\begin{tikzcd}
\frPic_{g,\alpha,d}^{\log}(\sA) \arrow[d] \arrow[r]                   & {\frPic_{g,n,d}} \arrow[d] \\
{\frM_{g,\alpha}^{\log}(\sA)} \arrow[r] & {\frM_{g,n}}            
\end{tikzcd}
\end{equation*}
where $\frPic_{g,n,d}$ is the Picard stack over $\frM_{g,n}$ of degree $d$ line bundles. 
\end{definition}

\begin{definition}\label{def:basis_CFK_obs_th}
Let $X$ be a smooth projective toric variety with GIT presentation $\AAA^M\GIT\GG_m^s$ induced by the fan $\Sigma$. Let $D_{i}$ denote the divisor induced by the character of $\GG_m^s$ given by the $i^{\mathrm{th}}$ projection and let $d_i = \beta\cdot D_{i}$. 

We define
\begin{equation}\label{eq:log_Pic_fibre_product}
\frPic^{\log,s}_{g,\alpha,\underline{d}}(\sA):=\frPic^{\log}_{g,n,d_1}(\sA)\times_{\frM^{\log}_{g,n}(\sA)}\ldots\times_{\frM^{\log}_{g,n}(\sA)}\frPic^{\log}_{g,n,d_s}(\sA).
\end{equation}
Let $q:\frPic^{\log,s}_{g,\alpha,\underline{d}}(\sA)\to \frM^{\log}_{g,\alpha}(\sA)$ denote the natural projection.
\end{definition}

\begin{proposition}\label{prop:obstruction_theory_X_over_Pic}
Let $\rho_{0} \in \Sigma(1)$, and let $D_{\rho_{0}}$ denote the associated toric divisor. For $\rho \in \Sigma(1)$ let $\cO_{\cX}(D_{\rho})$ denote the associated line bundle on $\cX = [\AAA^M/\GG_m^s]$. Then 
\[
\EE_{X|D_{\rho_0}}:= \oplus_{\rho \neq \rho_{0} \in \Sigma(1)}\mathbf{R}\pi_*u^*\cO_{\cX}(D_{\rho})\oplus \mathbf{R}\pi_*\cO_{\cC}
\]
is a dual obstruction theory for $\nu:\cQ^{\log}_{g,\alpha}(X|D_{\rho_0},\beta)\to \frPic^{\log,s}_{g,\alpha,\underline{d}}(\sA)$. Moreover, this obstruction theory induces the same class as the one in \Cref{quasimaps obs}.
\end{proposition}

\begin{proof} The logarithmic Euler sequence ~\cite[Equation 13]{shafi2024logarithmic} induces a distinguished triangle
\[
 \cO_{\cX}^{\oplus s}\to
\oplus_{\rho \neq \rho_{0}}\cO_{\cX} (D_{\rho}) \oplus \cO_{\cX}\to 
\mathbb{T}_{\cX}^{\log}\to \cO_{\cX}^{\oplus s}[1].
\]
Taking the pullback via the universal map and then pushing down to the moduli space gives
\[
\mathbf{R}\pi_*u^*\cO_{\cX}^{\oplus s}\to \oplus_{\rho \neq \rho_{0}}\mathbf{R}\pi_*u^*\cO_{\cX}(D_{\rho})\oplus \mathbf{R}\pi_*\cO_{\cC}
 \to 
\mathbf{R}\pi_*\mathbf{L}u^*\mathbb{T}_{\cX}^{\log}.
\]
We see that we have a commutative diagram 
\[\xymatrix{
\mathbb{T}_{\cQ^{\log}_{g,\alpha}(X|D_{\rho_{0}},\beta)/\frM^{\log}_{g,\alpha}(\sA)}\ar[r]\ar[d]&\nu^*\mathbb{T}_{\frPic^{\log,s}_{g,\alpha, \underline{d}}(\sA)/\frM^{\log}_{g,n}(\sA)}\ar[d]\\
\mathbf{R}\pi_*\mathbf{L}u^*\mathbb{T}_{\cX}^{\log}\ar[r]&\mathbf{R}\pi_*\cO_{\cC}^{\oplus s}[1]
}
\]
This induces a morphism of distinguished triangles
\[\xymatrix{
\mathbb{T}_{\cQ^{\log}_{g,\alpha}(X|D_{\rho_0},\beta)/\frM^{\log}_{g,\alpha}(\sA)}\ar[r]\ar[d]&\nu^*\mathbb{T}_{\frPic^{\log,s}_{g,\alpha,\underline{d}}(\sA)/\frM^{\log,s}_{g,\alpha}(\sA)}\ar[d]\ar[r]\ar[d]&\mathbb{T}_{\cQ^{\log}_{g,\alpha}(X|D_{\rho_0},\beta)/\frPic^{\log,s}_{g,\alpha,\underline{d}}(\sA)}[1]\ar@{.>}[d]^{\psi}\\
\mathbf{R}\pi_*\mathbb{T}_{\cX}^{\log}\ar[r]&\mathbf{R}\pi_*\cO_{\cC}^{\oplus s}[1]\ar[r]& \mathbf{R}\pi_*\oplus_{\rho \neq \rho_0}u^*\cO_{\cX}(D_{\rho}) \oplus \mathbf{R}\pi_*\cO_{\cC}[1]
}
\]
Applying the four lemma to the long exact sequence in cohomology proves that $\psi$ is a (dual) obstruction theory.

By functoriality of virtual pull-backs applied to the composition
\[\begin{tikzcd}
\cQ^{\log}_{g,\alpha}(X|D_{\rho_0},\beta)\arrow[r,"\nu"]\ar[rr, bend left]&\frPic^{\log,s}_{g,\alpha,\underline{d}}(\sA)\arrow[r]&\frM^{\log}_{g,\alpha}(\sA)
\end{tikzcd}
\]
the result follows.
\end{proof}

Consider the stack $\frPic^{\log,s}_{g,n,\underline{d}}(\sA')$ obtained by replacing $\sA$ by $\sA'$ in \Cref{def:Pic_curly_A} and taking $X$ to be a double point degeneration in \Cref{def:basis_CFK_obs_th}. The following result has a proof analogous to that of \Cref{prop:obstruction_theory_X_over_Pic}.

\begin{proposition}\label{prop:POT_W_relative_Pic}
Let $W\to\AAA^1$ be a toric double point degeneration. Let $\rho_1,\rho_2\in\Sigma_W(1)$ correspond to the pieces of the central fibre as in \Cref{notation:degenerations}.
The complex 
\[
\mathbf{R}\pi_\ast \mathbf{L}u^\ast [
\begin{tikzcd}
    \oplus_{\rho\neq \rho_1,\rho_2\in\Sigma_W(1)}\cO_{\cW} (D_{\rho})  \oplus_{i=1}^{2} \cO_{\cW} \arrow{r}{f_2} & \cO_{\cW}
\end{tikzcd}
]
\]
is a dual obstruction theory for 
$\cQ^{\log}_{g,n}(W/\AAA^1,\beta)\to \frPic^{\log,s}_{g,n,\underline{d}}(\sA')$. Moreover, this obstruction theory induces the same class on $\cQ^{\log}_{g,n}(W/\AAA^1,\beta)$ as the one constructed in \Cref{prop : perfectness}.
\end{proposition}

\section{The degeneration formula}

In this section we prove the degeneration formula for quasimaps to a toric double point degeneration. We do this in two steps, decomposition and gluing. The steps follow \cite{abramovich2020decomposition} and then the arguments of \cite{kim2021degeneration}.

First we need to verify that the general fibre of $\cQ^{\log}_{g,n}(W/\AAA^1,\beta)$ recovers quasimaps to the general fibre of $W \rightarrow \AAA^1$.
\begin{lemma}\label{lemma : absolutegenericfibre}
    Let $W \rightarrow \AAA^1$ be a toric double point degeneration, let $t\in\AAA^1$ with $t\neq 0$ and let $\iota_t\colon t\to \AAA^1$ be the corresponding embedding. Then the following diagram is Cartesian
    \begin{equation*}
        \begin{tikzcd}
        {\cQ_{g,n}(X,\beta)} \arrow[d] \arrow[r] & {\cQ_{g,n}(W/\AAA^1,\beta)} \arrow[d] \\
        t \arrow[r,"\iota_t"]                            & \AAA^1               
        \end{tikzcd}    
    \end{equation*}
\end{lemma}

\begin{proof}
    A quasimap to $W$ consists of line-bundle section pairs $(L_\rho,s_\rho)$ for $\rho\in \Sigma_W(1)$. Let $\rho_1,\rho_2$ correspond to the pieces of the central fibre. An element of $\cQ_{g,n}(W/\AAA^1,\beta)\times_{\AAA^1} t$ is an element of $\cQ_{g,n}(W/\AAA^1,\beta)$ such that $s_{\rho_1}\otimes s_{\rho_2}$ is constant and equal to $t$, under the trivialisation of $L_{\rho_1}\otimes L_{\rho_2}$. This is enough to determine the pairs $(L_{\rho_1},s_{\rho_1})$ and $(L_{\rho_2},s_{\rho_2})$ up to a $\GG_m$-action. The remaining data is precisely the data of a quasimap to $X$. 
\end{proof}

\begin{lemma}\label{Cor : deformation invariance}
    Let $W \rightarrow \AAA^1$ be a toric double point degeneration, let $t\in\AAA^1$ with $t\neq 0$ and let $\iota_t\colon t\to \AAA^1$ be the corresponding embedding. Then the following diagram is Cartesian
       \begin{equation*}
    \begin{tikzcd}
     {\cQ_{g,n}(X,\beta)} \arrow[d] \arrow[r] & {\cQ_{g,n}^{\log}(W/\AAA^1,\beta)} \arrow[d] \\
t \arrow[r,"i_t"]                              & \AAA^1                         
    \end{tikzcd}
    \end{equation*}
Moreover, we have that $i_t^![\cQ^{\log}_{g,n}(W/\AAA^1,\beta)]^{\vir} = [\cQ_{g,n}(X,\beta)]^{\vir}$
\end{lemma}
\begin{proof}
    There is a Cartesian diagram
\begin{equation*}
    \begin{tikzcd}
\GG_m \arrow[d] \arrow[r] & \sA' \arrow[d] \\
\GG_m \arrow[r]           & \AAA^1        
\end{tikzcd}
\end{equation*}
Since the logarithmic structure on $\GG_m$ is trivial this induces a Cartesian diagram 
\begin{equation*}
    \begin{tikzcd}
{\frM_{g,n} \times \GG_m} \arrow[d] \arrow[r] & {\frM_{g,n}^{\log}(\sA'/\AAA^1)} \arrow[d] \\
\GG_m \arrow[r]                               & \AAA^1                                    
\end{tikzcd}
\end{equation*}
Combining \Cref{lemma : absolutegenericfibre} with the definition of $\cQ^{\log}_{g,n}(W/\AAA^1,\beta)$ implies the first claim. The second claim follows from the fact that the restriction of $\TT^{\log}_{\cW/\AAA^1}$ to any non-zero fibre is $\TT_{\cX}$ which follows from \cite[1.1 (iv)]{olsson2003logarithmic} applied to the diagram (in the logarithmic category)
\begin{equation*}
        \begin{tikzcd}
\cX \times \GG_m \arrow[d] \arrow[r] & {\cW} \arrow[d] \\
\GG_m \arrow[r]                               & \AAA^1.  
\end{tikzcd}
\end{equation*}
\end{proof}
\subsection{Decomposition}
\Cref{Cor : deformation invariance} tells us that the virtual class of $\cQ^{\log}_{g,n}(W/\AAA^1,\beta)$ restricts to $[\cQ_{g,n}(X,\beta)]^{\vir}$ on any non-zero fibre. By deformation invariance and \Cref{remark: vfc central fibre}, this can be related to $[\cQ_{g,n}^{\log}(X_0,\beta)]^{\vir} = {i_0}^{!}[\cQ^{\log}_{g,n}(W/\AAA^1,\beta)]^{\vir}$. Now we use the decomposition formula of \cite{abramovich2020decomposition} to decompose this class into a sum of classes coming from combinatorics of how the components of the curve behave.  

\begin{definition}\label{definition : bipartite graphs}
    Fix non-negative integers $g,n$. A \emph{decorated bipartite graph}, is the following data:
    \begin{itemize}
        \item a connected graph $\Gamma$ with vertex set $V(\Gamma)$ and edge set $E(\Gamma)$;
        \item a function $r : V(\Gamma) \rightarrow \{1,2\}$;
        \item for each vertex $V$, a choice $(n_V,g_V)$ where $g_V \geq 0$ and $n_V \subset \{1,\dots,n\}$;
        \item for each edge $E \in E(\Gamma)$, a positive integer $w_E \geq 0$, called the \emph{weight};
    \end{itemize}
subject to the following conditions:
\begin{enumerate}
    \item for an edge connecting two vertices $V_1$ and $V_2$, $r(V_1) \neq r(V_2)$;
    \item $\cup_V n_V = \{1,\dots, n\}$ and  $\cap_V n_V = \emptyset$;
    \item $\sum_V g_V + g(\Gamma)= g$.
\end{enumerate}
Abusing notation, we will refer to such a decorated bipartite graph by $\Gamma$.
\end{definition}

Recall that the perfect obstruction theory on $\cQ^{\log}_{g,n}(W/\AAA^1,\beta)$ from \Cref{prop : perfectness} is relative to $\frM_{g,n}^{\log}(\sA'/\AAA^1)$.
If $\sA_0 = \sA' \times_{\AAA^1} 0$, then there is a Cartesian tower
    \begin{equation}\label{diagram : Cartesian tower}
\begin{tikzcd}
{\cQ^{\log}_{g,n}(X_0,\beta)} \arrow[d] \arrow[r] & {\cQ^{\log}_{g,n}(W/\AAA^1,\beta)} \arrow[d] \\
{\frM_{g,n}^{\log}(\sA_0)} \arrow[r] \arrow[d]    & {\frM_{g,n}^{\log}(\sA'/\AAA^1)} \arrow[d]   \\
0 \arrow[r, "i_0"]                                                  & \AAA^1.                                     
\end{tikzcd}
\end{equation}

In Diagram \eqref{diagram : Cartesian tower}, $\frM^{\log}_{g,n}(\sA_0)$ refers to minimal prestable logarithmic maps to $\sA_0$ \emph{over the standard logarithmic point}.  A consequence of Diagram \eqref{diagram : Cartesian tower} is that the perfect obstruction theory defining the virtual class $[\cQ^{\log}_{g,n}(X_0,\beta)]^{\vir} = i_{0}^![\cQ^{\log}_{g,n}(W/\AAA^1,\beta)]^{\vir}$ is relative to $\frM^{\log}_{g,n}(\sA_0)$. In \cite{abramovich2020decomposition}, the authors show that the fundamental class of $\frM_{g,n}^{\log}(\sA_0)$ decomposes into a sum of contributions coming from decorated bipartite graphs (introduced in \Cref{definition : bipartite graphs}). We denote by $\mathrm{Aut}(\Gamma)$ the subset of automorphisms of the underlying graph $\Gamma$ that respect the extra data $r, (n_V, g_V)_V$ and $(\omega_E)_E$. We let $m_\Gamma\in \NN\setminus \{0\}$ denote the projection of the generator of the monoid $Q_\Gamma^\vee\simeq \NN$ to $\Sigma(\mathrm{pt}_{\NN}) \simeq \RR_{\geq 0}$.   More explicitly, $m_\Gamma$ is the least common multiple of the weights $\{\omega_E\}$.

\begin{theorem}\label{Theorem : Universal Decomposition} ~\cite[Corollary 3.8 $\&$ Lemma 3.9]{abramovich2020decomposition}
\begin{equation*}
       [\frM_{g,n}^{\log}(\mathscr{A}_0)] = \sum_{\Gamma} \frac{m_{\Gamma}}{|\mathrm{Aut}(\Gamma)|} i_{\Gamma, *}[\frM_{g,n}^{\log}(\mathscr{A}_0,\Gamma)]
\end{equation*}
\end{theorem}

As in the case of stable maps ~\cite[Theorem 3.11]{abramovich2020decomposition} we will use \Cref{Theorem : Universal Decomposition} to decompose the virtual class $[\cQ^{\log}_{g,n}(X_0,\beta)]^{\vir}$ into contributions coming from quasimap moduli spaces associated to bipartite graphs with extra information about the quasimap degree on each curve component. We first make the formal definition.

\begin{definition}
For each decorated bipartite graph $\Gamma$ we define $\cQ^{\log}_{g,n}(X_0,\Gamma)$ via the Cartesian diagram
\begin{equation}\label{define : qlogGamma}
    \begin{tikzcd}
{\cQ^{\log}_{g,n}(X_0, \Gamma)} \arrow[d,"q"] \arrow[r,"j_{\Gamma}"] & {\cQ^{\log}_{g,n}(X_0, \beta)} \arrow[d,"p"] \\
{\frM_{g,n}^{\log}(\mathscr{A}_0,\Gamma)} \arrow[r,"i_{\Gamma}"]                           & {\frM_{g,n}^{\log}(\mathscr{A}_0)}.      
\end{tikzcd}
\end{equation}
\end{definition}
\begin{lemma}\label{Lemma : scheme theoretic gluing quasimaps}
    The space $\cQ^{\log}_{g,n}(X_0,\Gamma)$ decomposes as a disjoint union $\cup_{\widetilde{\Gamma} = (\Gamma,(\beta_V)_V)} \cQ^{\log}_{g,n}(X_0,\widetilde{\Gamma})$ where each $\beta_V$ is a quasimap degree of $W$ and 
    a point of $\cQ^{\log}_{g,n}(X_0,\widetilde{\Gamma})$ is
\begin{itemize}
    \item an $n$-pointed genus $g$ logarithmic quasimap from $C/\mathrm{pt}_Q$ to $W/\AAA^1$ which factors through the zero fibre, i.e. an element of $\cQ_{g,n}^{\log}(X_0,\beta)$;
    \item for each vertex $V$, a stable quasimap to $X_{r(V)}$, where the GIT presentation is induced from the embedding in $W$, from a curve $C_V$ with marked points $x_e$ for each edge containing $V$;
\end{itemize}
such that the gluing of the curves $C_V$ along the points corresponding to the edges in $E(\widetilde{\Gamma})$ give the underlying curve of $C$. Furthermore, the underlying quasimap of the element of $\cQ_{g,n}^{\log}(X_0,\beta)$ is \emph{glued} from the quasimaps from each vertex.

\end{lemma}

\begin{proof}
      An element of $\cQ_{g,n}^{\log}(X_0,\Gamma)$ comes with the data of a decomposition of the underlying curve $C$ defining the quasimap into $C_V$ for each $V \in V(\Gamma)$, with marked points $x_e$ corresponding to the edge containing $V$ and morphisms $C_V \rightarrow \sA$ such that the gluing the $C_V$ according to the graph and gluing the morphisms $C_V \rightarrow \sA$ produces the underlying morphism of $C \rightarrow \sA_0$ induced by the logarithmic quasimap. On the other hand there is a Cartesian diagram 
\begin{equation*}
    \begin{tikzcd}
\cX_{r(V)} \arrow[d] \arrow[r] & \cW_0 \arrow[d] \arrow[r] & \cW \arrow[d] \\
\sA \arrow[r]      & \sA_0 \arrow[r]           & \sA'         
\end{tikzcd}
\end{equation*}
where $\cX_{r(V)}$ is the quotient stack inside of which $X_{r(V)}$ lives when the GIT presentation is induced by the embedding in $W$. Therefore the morphism $C_V \rightarrow \sA$ induces a morphism $C_V \rightarrow \cX_{r(V)}$, i.e. a quasimap to $X_{r(V)}$.
\end{proof}

Diagram \eqref{define : qlogGamma} implies that the spaces $\cQ^{\log}_{g,n}(X_0,\widetilde{\Gamma})$ admit a perfect obstruction theory over $\frM^{\log}_{g,n}(\sA_{0},\Gamma).$

\begin{proposition}\label{Prop: decomposition}
    $$[\cQ_{g,n}^{\log}(X_0,\beta)]^{\vir} = \sum_{\widetilde{\Gamma}= (\Gamma,(\beta_V)_V)} \frac{m_{\Gamma}}{|\mathrm{Aut}(\Gamma)|} j_{\Gamma, *}[\cQ_{g,n}^{\log}(X_0,\widetilde{\Gamma})]^{\vir}.$$
\end{proposition}

\begin{proof}
We have that
\begin{align*}
   &\, [\cQ_{g,n}^{\log}(X_0,\beta)]^{\vir}  \\
  &=  p^{!}[\frM_{g,n}^{\log}(\mathscr{A}_0/0)] \\
   &= \sum_{\Gamma} \frac{m_{\Gamma}}{|\mathrm{Aut}(\Gamma)|} p^{!}i_{\Gamma, *}[\frM_{g,n}^{\log}(\mathscr{A}_0/0,\Gamma)]  \\
    &= \sum_{\widetilde{\Gamma} = (\Gamma, (\beta_V)_V)} \frac{m_{\Gamma}}{|\mathrm{Aut}(\Gamma)|} j_{\Gamma, *}[\cQ_{g,n}^{\log}(X_0,\widetilde{\Gamma})]^{\vir}.    \qedhere
\end{align*}
\end{proof}
\begin{remark}
  One may wonder why in \Cref{Lemma : scheme theoretic gluing quasimaps}, the data of an element of $\cQ^{\log}_{g,n}(X_0,\widetilde{\Gamma})$ is a quasimap to $X_{r(V)}$ \emph{with presentation induced by the embedding in $W$}. In fact $X_{r(V)}$ is a toric variety and its fan induces a GIT presentation. On the other hand a quasimap to $X_{r(V)}$ with this toric presentation contains strictly less information and hence \Cref{Lemma : scheme theoretic gluing quasimaps} would not be true with this presentation. We explain this phenomenon in \Cref{subsec:quasimaps_special_fibre}, see \Cref{ex:P1_relative_infinity}.
\end{remark}

\begin{remark}
    Analogously to \Cref{cfk obs} we have that the obstruction theory of $\cQ^{\log}_{g,n}(X_0,\widetilde{\Gamma})$ can be defined relative to the fibre product
\begin{equation*}
    \begin{tikzcd}
\frPic_{g,n,\underline{d}}^{\log}(\sA_0,\Gamma) \arrow[d] \arrow[r]                   & {\frPic_{g,n,\underline{d}}}(\sA_0) \arrow[d] \\
{\frM_{g,n}^{\log}(\sA_0,\Gamma)} \arrow[r] & {\frM_{g,n}^{\log}(\sA_0)}            
\end{tikzcd}
\end{equation*}
\end{remark}
\subsection{Gluing}
Fix a stable decorated bipartite graph $\tilde{\Gamma}$. Now we relate $[\cQ_{g,n}^{\log}(X_0,\widetilde{\Gamma})]^{\vir}$ to the logarithmic moduli spaces corresponding to the vertices of $\widetilde{\Gamma}$. More precisely, let $\Gamma_V$ denote the graph with a single vertex, with half edges for any edge with vertex $V$ recording the degree $\beta_V$, genus $g_V$ as well as the contact order at each half edge $\alpha_V = (w_{E} : V \in E)$. For each vertex we can form the moduli space $\cQ_{g_V,\alpha_V}^{\log}(X_V|D,\beta_V)$ as in \cite{shafi2024logarithmic}.

\begin{remark}
    Strictly speaking, we need an extension of the moduli spaces defined in \cite{shafi2024logarithmic} to allow for the presentation induced here by the degenerations (as well as non-proper targets), but this can be done and is addressed in~\cite[End of Section 1]{shafi2024logarithmic}.
\end{remark}
Observe that for each edge $E$ of $\tilde{\Gamma}$ there are two natural maps
$$\prod_{V}\cQ_{g_V,\alpha_V}^{\log}(X_V|D,\beta_V) \rightarrow D$$ because each edge is adjacent to two vertices. We build a map $$\prod_{V}\cQ_{g_V,\alpha_V}^{\log}(X_V|D,\beta_V) \rightarrow \prod_{E}D^2.$$
Define $\bigtimes_V \cQ_V$ via the Cartesian diagram
\begin{equation}\label{eq:definition_cross_V}
    \begin{tikzcd}
\bigtimes_V \cQ_V \arrow[d] \arrow[r]      & {\prod_{V \in \Gamma} \cQ^{\log}_{g_V,\alpha_V}(X_V|D,\beta_V)} \arrow[d] \\
\prod_{E} D \arrow[r, "\Delta"] & \prod_{E} D^2                            
\end{tikzcd}
\end{equation}
\begin{proposition}
    There is a morphism $$c_{\widetilde{\Gamma}} :\cQ^{\log}_{g,n}(X_0,\tilde{\Gamma}) \rightarrow \bigtimes_V \cQ_V$$ which is étale of degree $\frac{\prod_{E} w_E}{\mathrm{lcm}(w_E)_E}$. 
\end{proposition}

\begin{proof}
    The existence of the morphism follows from the description of $\cQ^{\log}(X_0,\tilde{\Gamma})$ of \Cref{Lemma : scheme theoretic gluing quasimaps}. Since nodes are not basepoints, the rest of the result follows from the analogous statement for stable maps which is proved in ~\cite[Lemma 9.2]{kim2021degeneration}.
\end{proof}
\begin{theorem}[Gluing]\label{Thm : gluing}
    $$[\cQ^{\log}_{g,n}(X_0,\tilde{\Gamma})]^{\vir} = c_{\widetilde{\Gamma}}^*\Delta^{!}\prod_{V} [\cQ^{\log}_{g_V,\alpha_V}(X_{r(V)}|D,\beta_V)]^{\vir}$$
\end{theorem}
As in \cite{kim2021degeneration, bousseau2019tropical}, \Cref{Thm : gluing} will be proved by showing that both classes come from compatible perfect obstruction theories with respect to different bases. 

In the following we fix a decorated bipartite graph $\widetilde{\Gamma}$ and use the following simplified notation.
Let $\cQ_{V}$ denote $\cQ^{\log}_{g_V,\alpha_V}(X_{r(V)}|D,\beta_V)$. Similarly, we denote 
\begin{align}
\cQ_{\widetilde{\Gamma}} = \cQ^{\log}_{g,n}(X_0,\tilde{\Gamma}), \qquad &\prod_{V} \cQ_{V} = \prod_{V}\cQ^{\log}_{g_V,\alpha_V}(X_{r(V)}|D,\beta_V), \label{eq:Q_gamma_tilde}\\  
\frM_{\Gamma} = \frM_{g,n}^{\log}(\mathscr{A}_0,\Gamma), \qquad &\prod_V \frM_V = \prod_{V} \frM^{\log}_{g_V,\alpha_V}(\sA).
\end{align}
Let $\pi_V:\cC_V\to \cQ_{V}$ be the universal curve over $\cQ_V$, and let $\cC_{V,\widetilde{\Gamma}}$ denote the pieces of the partial normalisation of $\cC$ along the nodes corresponding to the edges of $\widetilde{\Gamma}$. With this, we have a commutative diagram
\begin{equation}
\begin{tikzcd}\label{diagram : big}
\cW                                          &                                                                   & \cX_{r(V)} \arrow[ll, "i_{r(V)}"]           \\
\cC \arrow[d, "\pi"] \arrow[u, "u"]        & {\cC_{V,\widetilde{\Gamma}}} \arrow[l, "j_V"] \arrow[r, "\ell_V"] & \cC_V \arrow[d, "\pi_V"] \arrow[u, "u_V"] \\
\cQ_{\widetilde{\Gamma}} \arrow[rr, "g_V"] &                                                                   & \cQ_V                                    
\end{tikzcd}
\end{equation}
as well as 
\begin{equation*}
    \begin{tikzcd}
\cQ_{\widetilde{\Gamma}} \arrow[d, "i_E", hook] \arrow[r, "\ev_E"] & {D \subset \cD} \arrow[d, "i", hook] \\
\cC \arrow[r, "u"]                                        & \cW                   
\end{tikzcd}
\end{equation*}
where $\ev_E$ and $i_E$ utilise the fact that the universal curve has a distinguished node for the edge $E$. 

Recall that by \Cref{quasimaps obs}, for each vertex $V \in V(\widetilde{\Gamma})$, the complex $(\mathbf{R}\pi_{V,*}u_V^*\TT^{\log}_{\cX_{r(V)}})^{\vee}$ is a perfect obstruction theory for the morphism $$\cQ^{\log}_{g_V,\alpha_V}(X_{r(V)}|D,\beta_V)\to\frM_{g_V,\alpha_V}^{\log}(\sA).$$

Let $\cE$ denote the complex $u^*\TT^{\log}_{\cW/\AAA^1}$ on the universal curve $\cC_{\widetilde{\Gamma}}$ over $\cQ^{\log}_{g,n}(X_0,\tilde{\Gamma})$. Let $$\mathbb{E}=(\mathbf{R}\pi_*\cE)^{\vee}.$$ By \Cref{prop : perfectness}, $\mathbb{E}$ defines a perfect obstruction theory of $\cQ^{\log}_{g,n}(X_0,\tilde{\Gamma})$ relative to $\frM_{g,n}^{\log}(\mathscr{A}_0,\Gamma)$. 
\begin{lemma}\label{seq e}
    There is a distinguished triangle
\begin{equation}\label{exact seq : E}
\cE \rightarrow  \oplus_{V \in V(\widetilde{\Gamma})}j_{V,*} \ell_V^*
u_V^*\TT^{\log}_{\cX_{r(V)}} \rightarrow \oplus_{E \in E(\widetilde{\Gamma})} i_{E,*}\ev_E^*\TT^{\log}_{\cW/\AAA^1}|_{D} \rightarrow \cE[1]
\end{equation}
where $i_E : \cQ_{\widetilde{\Gamma}} \rightarrow \cC_{\widetilde{\Gamma}}$ is the inclusion of the node corresponding to $E$ and and $\ev_E : \cQ_{\widetilde{\Gamma}} \rightarrow D$ is the evaluation at the node.
\end{lemma}

\begin{proof}
There is a partial normalisation exact sequence on the universal curve  $\cC_{\widetilde{\Gamma}}$ given by 
$$0 \rightarrow \cO_{\cC_{\widetilde{\Gamma}}} \rightarrow \oplus_{V \in V(\widetilde{\Gamma})} j_{V,*}\cO_{\cC_{V,\widetilde{\Gamma}}} \rightarrow \oplus_{E \in E(\widetilde{\Gamma})} i_{E,*}\cO_{\cQ_{\widetilde{\Gamma}}} \rightarrow 0$$
We claim the result follows from tensoring the associated distinguished triangle with $\cE$. We have that $i_{E,*}\cO_{\cQ_{\widetilde{\Gamma}}}\otimes u^*\TT^{\log}_{\cW/\AAA^1} \cong i_{E,*}(i_{E}^* u^* \TT^{\log}_{\cW/\AAA^1}) \cong i_{E,*}\ev_E^* \TT^{\log}_{\cW/\AAA^1}|_D$. For the middle term we can notice from our explicit description above that $\TT^{\log}_{\cW/\AAA^1}|_{\cX_{r(V)}} \cong \TT^{\log}_{\cX_{r(V)}}$ using the explicit description for the latter given in ~\cite[3.5.20]{shafi2022enumerative}. We have 
\[
j_{V,*} \cO_{\cC_{V,\widetilde{\Gamma}}} \otimes u^* \TT^{\log}_{\cW/\AAA^1} \cong j_{V,*}j_{V}^*u^*\TT^{\log}_{\cW/\AAA^1} \cong j_{V}^* \ell_V^* u_V^* \TT^{\log}_{\cX_{r(V)}}
\]
by the projection formula and commutativity of Diagram \eqref{diagram : big}. This proves the claim.
\end{proof}
\begin{lemma}
We have a morphism 
\begin{equation}\label{map obs}
    \oplus_{V \in V(\widetilde{\Gamma})} 
 j_{V,*}\ell_{V}^*u_V^*\TT_{\cX_{r(V)}}^{\log} \rightarrow \oplus_{E \in E(\widetilde{\Gamma})}i_{E,*}\ev_E^* T_D.
\end{equation}
\end{lemma}
\begin{proof}
    Let $\cD_\rho$ be the toric divisors of $\cX_{r(V)}$ and without loss of generality we assume that $\cD=\cD_{\rho_1}$. The logarithmic tangent complex of $\cX_{r(V)}$ is  
 \[\cO^{\oplus s}_{\cX_{r(V)}}\to \oplus_{\rho\neq \rho_1} 
\cO_{\cX_{r(V)}}(\cD_\rho)\oplus\cO_{\cX_{r(V)}}
 \]
 and the tangent complex of $\cD$ is
\[\cO_{\cD}^{\oplus s}\to \oplus_{\rho\neq \rho_1} \cO_{\cD}(\cD_\rho)
\]
To construct the morphism \eqref{map obs} it suffices to construct the morphism for $\TT^{\log}_{\cX_{r(V)}}$ and $T_{D}$ replaced with either the first or second terms in the above complexes. Consider the diagram
\begin{equation*}
    \begin{tikzcd}
\cC                                                                            & {\cC_{V,\widetilde{\Gamma}}} \arrow[l, "j_V"', hook] \arrow[r, "\ell_V"] \arrow[rd, "u'_V"] & \cC_V \arrow[d, "u_V"] \\
\cQ_{\widetilde{\Gamma}} \arrow[r, "\ev_E"] \arrow[u, "i_E"] \arrow[ru, "j^E_V"] & D \arrow[r, hook]                                                                           & \cX_{r(V)}            
\end{tikzcd}
\end{equation*}
For the second term in the complex (the first is similar) we see that
\[
j_{V,*}{u'}^*_V \left(\oplus_{\rho\neq \rho_1}\cO_{\cX_{r(V)}}(\cD_\rho)\oplus\cO_{\cX_{r(V)}}\right) \rightarrow j_{V,*}{u'}^*_V \left(\oplus_{\rho\neq \rho_1} \cO_{\cD}(\cD_\rho)\right)
\]
by the restriction and the projection. And then we have a morphism 
\[
j_{V,*}{u'}^*_V \left(\oplus_{\rho\neq \rho_1} \cO_{\cD}(\cD_\rho)\right) \rightarrow j_{V,*}j^E_{V,*} j_{V}^{E,*}{u'}^*_V \left(\oplus_{\rho\neq \rho_1} \cO_{\cD}(\cD_\rho)\right) = i_{E,*}\ev_{E}^*\left(\oplus_{\rho\neq \rho_1} \cO_{D}(D_\rho)\right)
\]
By adjunction and commutativity of the diagram. This procedure builds a morphism like in \eqref{map obs}, but only for one summand of the source and target for a vertex $V$ contained in an edge. We form the morphism \eqref{map obs} by taking the difference of sections for every pair of vertices contained in an edge. 
\end{proof}
\begin{notation}\label{cone}
    Let $\cF$ be defined by the following distinguished triangle
     \[\cF \rightarrow \oplus_{V\in V(\widetilde{\Gamma})} j_{V,*} \ell_{V}^* u_V^*\TT_{\cX_{r(V)}}^{\log}  \rightarrow \oplus_{E(\widetilde{\Gamma})}i_{E,*}\ev_E^* T_{D} \rightarrow \cF[1]
     \]
    on the universal curve over $\cQ_{\widetilde{\Gamma}}$. Let $\mathbb{F}:=(\mathbf{R}\pi_*\cF)^{\vee}$. 
\end{notation}

\begin{lemma}
   In notation as above, we have that $\mathbb{F}$ is a perfect obstruction theory for the morphism 
   \[\cQ^{\log}_{g,n}(X_0,\widetilde{\Gamma}) \to\prod_{V} \frM^{\log}_{g_V,\alpha_V}(\sA).
   \]
   The induced class is $c_{\widetilde{\Gamma}}^*\Delta^!\prod_{V} [\cQ^{\log}_{g_V,\alpha_V}(X_V|D,\beta_V)]^{\vir}$.
\end{lemma}
\begin{proof}
    Consider the diagram
\begin{equation*}
    \begin{tikzcd}
\cQ_{\widetilde{\Gamma}} \arrow[r,"c_{\widetilde{\Gamma}}"]\arrow[rr, "g", bend left] & \bigtimes_V \cQ_V \arrow[r,"\Delta_{\cQ}"] \arrow[d, "\ev_E"]                          & {\prod_{V}\cQ_V} \arrow[r] \arrow[d] & {\prod_{V}\frM_V} \\
            & \prod_{E \in E(\widetilde{\Gamma})} D_E \arrow[r, "\Delta"] & \prod_{E \in E(\widetilde{\Gamma})} D_E^2                                                              &                                                                     
\end{tikzcd}
\end{equation*}
The map $g:\cQ_{\widetilde{\Gamma}} \rightarrow \prod_V \cQ_V$ is the composition of the \'etale map $c_{\widetilde{\Gamma}}$ and $\Delta_{\cQ}$, so it admits a perfect obstruction theory induced by the regular embedding $\Delta$. Let
\[\mathbb{B}^{\vee}:= \oplus_{E \in E(\widetilde{\Gamma})} \ev_E^*[0 \rightarrow  T_{D}]
\]
in degrees $[0,1]$. Then $\mathbb{B}$ is a perfect obstruction theory for $g$. We know that $\prod_V \cQ_V \rightarrow \prod_V \frM_V$ admits a dual perfect obstruction theory 
\[
\mathbb{A}^{\vee}:=\oplus_{V}\mathbf{R}\pi_{V,*}u_{V}^*\TT^{\log}_{\cX_{r(V)}}. 
\]
In the following we use $\mathbb{A}$ and $\mathbb{B}$ to construct a perfect obstruction theory for $\cQ_{\widetilde{\Gamma}} \rightarrow \prod_{V} \frM_{V}$. 
We first prove that $g^*\mathbb{A}^{\vee}\cong \oplus_{V \in V(\widetilde{\Gamma})}\mathbf{R}\pi_*j_{V,*}\ell_V^* u_V^* \TT^{\log}_{\cX_{r(V)}}$. Consider the following commutative diagram with the  square being Cartesian
\[
\begin{tikzcd}
\cC \arrow[rd, "\pi"']        & {\cC_{V,\widetilde{\Gamma}}} \arrow[l, "j_V"'] \arrow[d, "\pi_{V,\widetilde{\Gamma}}"]\arrow[r, "\ell_V"] & \cC_V \arrow[d, "\pi_V"] \\
&\cQ_{\widetilde{\Gamma}} \arrow[r, "g_V"]                                                                   & \cQ_V.                                    
\end{tikzcd}
\]
The statement follows from $\pi_{V,\widetilde{\Gamma},*}=\pi_*j_{V,*}$ and cohomology and base change for $u_V^* \TT^{\log}_{\cX_{r(V)}}$ in the above square. 

Note that $\pi_*i_{E,*}\ev_E^* T_D=\ev^*_ET_D$. With these we can apply \Cref{map obs} to get a morphism $g^*\mathbb{A}^{\vee}[-1] \to \mathbb{B}^{\vee}$. By the definition of $\mathbb{F}$ in \Cref{cone}, we have a distinguished triangle
 \[\mathbb{F} \rightarrow \mathbb{B}  \rightarrow g^*\mathbb{A}[1] \rightarrow \mathbb{F}[1]
    \]
    Moreover, we obtain a morphism of distinguished triangles
\begin{equation*}
\begin{tikzcd}
\mathbb{B} \arrow[d] \arrow[r]                              & g^*\mathbb{A}[1] \arrow[d] \arrow[r]                                              & \mathbb{F}[1] \arrow[d] \arrow[r] & {} \\
\mathbb{L}_{{\cQ_{\tilde{\Gamma}}}/\prod \cQ_{V}} \arrow[r] & g^*\mathbb{L}_{\prod \cQ_V/\prod \frM_{\Gamma}}[1] \arrow[r] & {\mathbb{L}_{\cQ_{\widetilde{\Gamma}}/\prod \frM_V}[1]} \arrow[r]                  & {}
\end{tikzcd}
\end{equation*}
where the commutativity of left square above is commutativity of Diagram 9.10 in \cite{kim2021degeneration}. A repeated application of the Four Lemma implies that $\mathbb{F}$ is an obstruction theory for $\cQ_{\widetilde{\Gamma}} \rightarrow \prod_{V} \frM_{V}$. Moreover, the diagram above shows that we have a compatible triple of obstruction theories. By functoriality of pull-backs we have that the class induced by $\mathbb{F}$ is equal to the class 
\[c_{\widetilde{\Gamma}}^*\Delta^!\prod_{V} [\cQ^{\log}_{g_V,\alpha_V}(X_V|D,\beta_V)]^{\vir}. \qedhere
\]
\end{proof}

We have morphisms
\begin{equation*}
    \begin{tikzcd}
\cQ_{\tilde{\Gamma}} \arrow[r] & \frM_{\Gamma} \arrow[r, "\delta"] & \prod_V \frM_{V}.
\end{tikzcd}
\end{equation*}

\begin{proposition}
Let $\mathbb{E}^{\vee}$ be the complex $\mathbf{R}\pi_*u^*\TT^{\log}_{\cW/\AAA^1}$.
    The dotted arrow in the following diagram can be filled in to form a commutative diagram in the derived category of $\cQ_{\tilde{\Gamma}}$.
\begin{equation*}
\begin{tikzcd}
\mathbb{F} \arrow[d] \arrow[r, dashed]                              & \mathbb{E} \arrow[d] \arrow[r]                                              & \mathbb{L}_{\delta}[1] \arrow[d, Rightarrow] \arrow[r] & {} \\
\mathbb{L}_{{\cQ_{\tilde{\Gamma}}}/\prod \frM_{V}} \arrow[r] & \mathbb{L}_{{\cQ_{\tilde{\Gamma}}}/ \frM_{\Gamma}} \arrow[r] & {\mathbb{L}_{\delta}[1]} \arrow[r]                  & {}
\end{tikzcd}
\end{equation*}
\end{proposition}
\begin{proof}
If we take the pushforward of Equation \eqref{exact seq : E} under $\pi$ we obtain the distinguished triangle
    \begin{equation} \label{seq:e}
    \mathbf{R}\pi_*u^*\TT^{\log}_{\cW/\AAA^1} \rightarrow \oplus_{V\in V(\widetilde{\Gamma})} \mathbf{R}\pi_* j_{V,*} \ell_V^* u_V^*\TT^{\log}_{\cX_V}\rightarrow \oplus_{E\in E(\widetilde{\Gamma})}\ev_E^*\TT^{\log}_{\cW/\AAA^1} |_D \rightarrow  
    \end{equation}
    
To construct the dotted arrow we use the definition of $\mathbb{F}$, the distinguished triangle \eqref{seq:e} and \Cref{map obs}. More precisely, we have solid arrows which induce the dashed arrow making the following diagram commute
\begin{equation*}
\begin{tikzcd}
\mathbb{E}^{\vee} \arrow[d, dashed] \arrow[r]                              & \oplus_{V\in V(\widetilde{\Gamma})}  \mathbf{R}\pi_* j_{V,*} \ell_V^* u_V^*\TT^{\log}_{\cX_V} \arrow[d] \arrow[r]                                              & \oplus_{E\in E(\widetilde{\Gamma})}\ev_E^*T^{\log}_{\cW/\AAA^1}|_{D} \arrow[d] \arrow[r] & {} \\
\mathbb{F}^{\vee}\arrow[r] &  \oplus_{V\in V(\widetilde{\Gamma})}  \mathbf{R}\pi_* j_{V,*} \ell_V^* u_V^*\TT^{\log}_{\cX_V} \arrow[r] & { \oplus_{E\in E(\widetilde{\Gamma})}\ev_E^*T_{D}} \arrow[r]                  & {}
\end{tikzcd}
\end{equation*}
We can complete the diagram
\begin{equation*}
\begin{tikzcd}
&&\oplus_{E\in E(\widetilde{\Gamma})}\ev_E^*\cO_{D}\arrow[d]\\
\mathbb{E}^{\vee} \arrow[d] \arrow[r]                              & \oplus_{V\in V(\widetilde{\Gamma})} \mathbf{R}\pi_* j_{V,*} \ell_V^* u_V^*\TT^{\log}_{\cX_V} \arrow[d] \arrow[r]                                              & \oplus_{E\in E(\widetilde{\Gamma})}\ev_E^*T^{\log}_{\cW/\AAA^1}|_{D} \arrow[d] \arrow[r] & {} \\
\mathbb{F}^{\vee}\arrow[r] &  \oplus_{V\in V(\widetilde{\Gamma})} \mathbf{R}\pi_* j_{V,*} \ell_V^* u_V^*\TT^{\log}_{\cX_V} \arrow[r] & { \oplus_{E\in E(\widetilde{\Gamma})}\ev_E^*T_{D}} \arrow[r]                  & {}
\end{tikzcd}
\end{equation*}
The octahedron axiom gives a distinguished triangle
\[
\mathbb{F}\to\mathbb{E}\to \oplus_{E\in E(\widetilde{\Gamma})}\ev_{E}^*\cO_{D}[1]\to \mathbb{F}[1].
\]
By \cite{kim2021degeneration}, Lemma 9.2, we have $\oplus_{E\in E(\widetilde{\Gamma})}\ev_{E}^*\cO_{D}\simeq \mathbb{L}_{\delta}$. The only thing left to prove is that the map $\mathbb{F}\to \mathbb{E}$ obtained above is the same obtained by completing the morphism $\mathbb{E}\to \mathbb{L}_{\delta}$ to a distinguished triangle. This is the commutativity of diagram 9.12 in \cite{kim2021degeneration}.
\end{proof}
\begin{proof}[Proof of \Cref{Thm : gluing}]
We have a compatible triple and by functoriality of pull-backs we get the statement. 
\end{proof}

The following is the main result of this section. Let $W \rightarrow \AAA^1$ be a toric double point degeneration such that the associated quasimap moduli space is proper. Let $X_0 = X_1 \cup_D X_2$ denote the central fibre.
\begin{theorem}[Degeneration formula] \label{Th: degeneration}
    \begin{equation*}
        [\cQ_{g,n}^{\log}(X_0,\beta)]^{\vir} = \sum_{\widetilde{\Gamma}} \frac{m_{\widetilde{\Gamma}}}{|\mathrm{Aut}(\widetilde{\Gamma})|} j_{\widetilde{\Gamma},*} c_{\widetilde{\Gamma}}^* \, \Delta^! \left(\prod_{V \in V(\widetilde{\Gamma})} [\cQ^{\log}_{g_V,\alpha_V}(X_{r(V)}|D,\beta_V)]^{\vir}\right).
    \end{equation*}
\end{theorem}
\begin{proof} The statement follows from \Cref{Prop: decomposition} together with \Cref{Thm : gluing}.
\end{proof}
\begin{remark}
    In Section \ref{sec:log-local_correspondence} we will use the degeneration formula with a certain point condition imposed on the moduli space directly. The proof of this result is analogous to this section. For details, see for example~\cite[Section 7.3]{bousseau2019tropical}.
\end{remark}
\section{Local/logarithmic correspondence}\label{sec:log-local_correspondence}

In this section we prove the local/logarithmic correspondence for quasimaps. Mirroring the Gromov--Witten version \cite{vangarrel2019local}, we will use the degeneration formula proved in the previous section to prove this correspondence for projective space and a coordinate hyperplane. We will then use this to pull back the local/logarithmic correspondence to any very ample set up. We will then show that stronger results hold than in the stable maps setting, see \Cref{snclocallog} and \Cref{qloglocalcorner}. The proof of these theorems will utilise the same strategy to prove these correspondences for any very ample set up.

\subsection{Smooth divisor}\label{subsec:log_local_smooth_divisor}
\begin{theorem}\label{qlocallog}
Let $X$ be a GIT quotient $W \GIT G$ and let $D \subset X$ be a \emph{smooth, very ample} divisor satisfying Assumptions \ref{assumptions:absolute_quasimaps} and \ref{assumptions:log_quasimaps}.
Let $\beta$ be an effective curve class on $X$ and let $d = D \cdot \beta \geq 0$, then we have the following equality of virtual classes
      $$[\cQ_{0,(d,0)}^{\log}(X|D,\beta)]^{\vir} = (-1)^{d+1} \cdot \ev_2^*D \cap [\cQ_{0,2}(\cO_{X}(-D),\beta)]^{\vir}.$$
\end{theorem}

As the degeneration formula for quasimaps in \Cref{Th: degeneration} only applies in a toric setting we will first use the degeneration formula to prove

\begin{theorem}\label{qlocallogtoric}
    Let $X=\PP^N$ for $N>1$ and $D=H$ a coordinate hyperplane, and let $d \in \NN$. Then
      $$[\cQ_{0,(d,0)}^{\log}(\PP^N|H,d)]^{\vir} = (-1)^{d+1} \cdot \ev_2^*H \cap [\cQ_{0,2}(\cO_{\PP^N}(-H),d)]^{\vir}.$$
\end{theorem}
Assuming \Cref{qlocallogtoric}, the proof of \Cref{qlocallog} is straightforward.

\begin{proof}[Proof of \Cref{qlocallog}]
We have a commutative diagram
    \begin{equation}
        \begin{tikzcd}
{\cQ^{\log}_{0,(d,0)}(X|D,\beta)} \arrow[d,"p"] \arrow[r,"i"] & {\cQ_{0,(d,0)}^{\log}(\PP^N|H,d)} \arrow[d,"p'"] \\
{\cQ_{0,2}(X,\beta)} \arrow[r,"i'"]                        & {\cQ_{0,2}(\PP^N,d)}                       
\end{tikzcd}
    \end{equation}
The existence of the horizontal arrows follows similarly from \cite[Section 3.1]{ciocan2014wall}. We can use the fact ~\cite[Proposition 3.5.10]{shafi2022enumerative} that the virtual class on $\cQ^{\log}_{0,(d,0)}(X|D,\beta)$ can be defined relative to $\cQ_{0,(d,0)}^{\log}(\PP^N|H,d)$ and is in fact pulled back from the obstruction theory for $i'$, which defines the virtual class on $\cQ_{0,2}(X,\beta)$. Let $\mathscr{L}$ (resp. $\mathscr{L}_D$) denote the universal line bundle on the universal curve $\cQ_{0,2}(\PP^N,d)$ (resp. $\cQ_{0,2}(X,\beta)$) corresponding to $\cO_{\PP^N}(H)$ (resp. $\cO_{X}(D)$). We then have that
\begin{align*}
    &p_*[\cQ^{\log}_{0,(d,0)}(X|D,\beta)]^{\vir} \\
    &= p_* (i')^{!}[\cQ_{0,(d,0)}^{\log}(\PP^N|H,d)]  \\
    &= {i'}^{!}{p'}_{*}[\cQ_{0,(d,0)}^{\log}(\PP^N|H,d)] \\
    &= {i'}^{!}(-1)^{d+1} \cdot \ev^*H \cap [\cQ_{0,2}(\cO_{\PP^N}(-H),d)]^{\vir} \\
    &= {i'}^{!}(-1)^{d+1} \cdot \ev^*H \cap e(\textbf{R}^1\pi_*\mathscr{L}^{\vee}) \cap [\cQ_{0,2}(\PP^N,d)] \\
    &= (-1)^{d+1} \cdot \ev^*D \cap e(\textbf{R}^1\pi_*\mathscr{L}_D^{\vee}) \cap [\cQ_{0,2}(X,\beta)]^{\vir} \\
    &= (-1)^{d+1} \cdot \ev^*D \cap [\cQ_{0,2}(\cO_X(-D),\beta)]^{\vir}. \qedhere
\end{align*}
\end{proof}

\begin{remark}
    Although we only prove \Cref{qlocallogtoric} for projective spaces of dimension greater than one, \Cref{qlocallog} implies the corresponding result for $\PP^1$. The reason for the restriction is that the for $\PP^1$ the induced presentation of the central fibre components are slightly more complicated so we exclude this analysis. On the other hand, the fact that the local/logarithmic correspondence holds for quasimaps for $\PP^1$ may be surprising. For example, a consequence is that the Gromov--Witten/quasimap wall-crossing for $\cO_{\PP^1}(-1)$
is trivial, after capping with a hyperplane class. This follows from the fact that the logarithmic spaces are birational. This is not at all an obvious conclusion from ~\cite[Theorem 4.2.1]{ciocan2017higher}. On the other hand there is precendent for this, in ~\cite[Section 9]{marian2011moduli} the authors show that quasimap and Gromov--Witten invariants of $\cO_{\PP^1}{(-1)}\oplus\cO_{\PP^1}(-1)$ coincide in all genera.
\end{remark}

In the remaining of \Cref{subsec:log_local_smooth_divisor}, we prove \Cref{qlocallogtoric}. The proof, which is deferred until the end of \Cref{subsec:log_local_smooth_divisor}, is a combination of the equalities in \Cref{Lemma : local pushdown} and \Cref{thm : 1 graph contribution,Prop : qlocal P1 comp}. Furthermore, \Cref{thm : 1 graph contribution} is the result of applying the degeneration formula (\Cref{Th: degeneration}) and ruling out all the contributions except for the graph 
\[
    \xymatrix{
    \overset{V_1}{\bullet}
    \ar@{-}[r]^(-.1){}="a"^(1.1){}="b" \ar^E@{-} "a";"b"
    &\overset{V_2}{\bullet}
    }
\] 
with $V_1$ an $L_1$-vertex, $V_2$ an $L_2$-vertex, $\omega_E = d$, $\beta_{V_1} = (d,-d)$, $\beta_{V_2}=(0,d)$ and exactly one marking on each vertex. To do so, we follow the next steps:
\begin{enumerate}
    \item \Cref{no several edges 1} shows there must be two vertices with a single bounded edge connecting them.
    \item \Cref{no non-fibers} shows there are no contributions from graphs with \emph{at least one} $L_2$-vertex with degree different from a multiple class of a fibre. 
    \item \Cref{no ghosts} shows there are no contributions from graphs with one edge, but ghost classes. This uses \Cref{ghost w smooth}, which follows from \Cref{p h smooth}.
    \item \Cref{no purely p classes} shows there are no contributions from graphs contained in $L_1$.
\end{enumerate}
This argument follows \cite{vangarrel2019local}, with two alterations.

The first alteration is minor. The main theorem of \cite{vangarrel2019local} compares the one marked logarithmic space with the unmarked local space and involves a forgetful morphism. Since neither of these spaces exist in quasimap theory we use the alteration to the degeneration argument discussed in \cite{fan2021WDVV,tseng2023mirror}, in addition to adding a redundant marking.

The second alteration requires an extra argument. On the one hand the stability condition of quasimap theory allows us to exclude certain graphs without effort. On the other hand we require an additional vanishing of contributions where the degrees contains  ``ghost'' curve classes in the sense of \Cref{def:ghost_class}. These occur because, as in \Cref{ex:P1_relative_infinity,example ghost 1}, the involved spaces of quasimaps depend on the GIT presentation of the target space. 

\begin{notation}\label{bundle def}
With the notations of \Cref{deformationPN}, let $W(\PP^N,H) = \Bl_{H \times 0} (\PP^N \times \AAA^1)$ be the deformation to the normal cone of $H$ in $\PP^N$ for $N>1$. The central fibre is a union of $X_1\simeq X = \PP^N$ and $X_2 \simeq \PP(\cO \oplus \cO(H)|_H)$ glued along $H$.

Let $L =\cO(-\widetilde{H})$ denote the total space of the line bundle over $W(\PP^N,H)$, where $\widetilde{H}$ is the strict transform of $H \times \AAA^1$. Over a non-zero point of $\AAA^1$ the fibre is just $\cO_{\PP^N}(-H)$. The central fibre $L_0$ is a union of two components $L_{1} = \PP^N \times \AAA^1$ and $L_{2}=\cO(-H_{\infty})$ over $X_2 \simeq \PP(\cO \oplus \cO(H)|_H)$, glued along $L_H = H \times \AAA^1$. We follow the convention from \Cref{notation:degenerations} to denote ambient stacks with calligraphic letters, such as $\cL$ for the stack corresponding to $L$ and $\cL_0$ for $L_0$. 

The degree of a quasimap to $W(\PP^N,H)$, with its toric presentation, is given by $\Pic^{\GG_m^2}\AAA^5 \simeq \ZZ^2$. We choose the same basis $\{l,f\}$ as in \Cref{deformationPN}. 
\end{notation}

\begin{remark}\label{remark : (d,-d) toric presentation}
    Although the degeneration formula is written in terms of logarithmic quasimap spaces with presentation induced by the total space we have that $$\cQ^{\log}_{g,\alpha}(\PP^N|H,(d,-d)) \simeq \cQ^{\log}_{g,\alpha}(\PP^N|H,d)$$ where the former space uses the presentation $\PP^N = \AAA^{N+2} \GIT \GG_m^2$ induced by $W(\PP^N,H)$ and the latter uses the standard presentation of $\PP^N$. The same is true for $L_1|L_H$.
\end{remark}

\begin{definition}\label{definition : infinity section}
  Let $\ev_2$ denote the evaluation at the second marked point. We define $\cQ^{\log}_{0,2}(L,(d,0))^{\widetilde{H}}$ as the fibre product 
\begin{equation*}
    \begin{tikzcd}
{\cQ^{\log}_{0,2}(L,(d,0))^{\widetilde{H}}} \arrow[d] \arrow[r] & {\cQ^{\log}_{0,2}(L,(d,0))} \arrow[d,"\ev_2"] \\
\widetilde{H} \arrow[r, "i_{\widetilde{H}}", hook]                          & W(\PP^N,H)                        
\end{tikzcd}
\end{equation*}
\end{definition}

The moduli space $\cQ^{\log}_{0,2}(L,(d,0))^{\widetilde{H}}$ comes with a virtual class given by Gysin pullback by $i_{\widetilde{H}}$. The following follows from the argument of ~\cite[Lemma 2.2]{vangarrel2019local}.

\begin{lemma}\label{Lemma : local pushdown}
There is a morphism $P : \cQ^{\log}_{0,2}(L,(d,0))^{\widetilde{H}} \rightarrow \cQ_{0,2}(\PP^N,d)$. Moreover, 
$$P_{*}[\cQ^{\log}_{0,2}(L_0,(d,0))^{\widetilde{H}}]^{\vir} = \ev_2^*H \cap [\cQ_{0,2}(\cO_{\PP^N}(-H),d)]^{\vir}.$$
\end{lemma}

Next we apply the degeneration formula. We will prove that the analogue of ~\cite[Theorem 2.3]{vangarrel2019local} also applies in this situation.

Let us first introduce the notation. Consider the following diagram 
\begin{equation}\label{prod-qmaps}
\begin{tikzcd}
\cQ^{\log}_{0,(d,0)}(L_1|L_H,d) \times_{L_H} \cQ^{\log}_{(d,0)}(L_2|L_H,(0,d))^{H_{\infty}} \arrow[d, "\mathrm{pr}_1"] & \cQ_{\widetilde{\Gamma}} \arrow[r, "j_{\widetilde{\Gamma}}"] \arrow[l, "c_{\widetilde{\Gamma}}"'] & {\cQ^{\log}_{0,2}(L,(d,0))^{\widetilde{H}}} \arrow[d, "P"] \\
{\cQ^{\log}_{0,(d,0)}(\PP^N|H,d)} \arrow[rr, "F"]                           &                                                                                                   & {\cQ_{0,2}(\PP^N,d)}.                         
\end{tikzcd}
\end{equation}
The notation $\cQ_{\widetilde{\Gamma}}$ was introduced in \Cref{eq:Q_gamma_tilde}. Here $\widetilde{\Gamma}$ is the graph with one $L_1$-vertex of degree $(d,-d)$, one $L_2$-vertex of degree $(0,d)$ (that is, $d$ times a fibre class), an edge of weight $d$ connecting the two vertices and an extra marking on each vertex. In addition there is a constraint to pass through $\widetilde{H}$ as in \Cref{definition : infinity section}, corresponding to the leg on the $L_2$-vertex. This graph determines the spaces $\cQ^{\log}_{0,(d,0)}(L_1|L_H,(d,-d))$ and $\cQ^{\log}_{0,(d,0)}(L_2|L_H,(0,d))^{H_{\infty}}$ by the degeneration formula. In general the former is a quasimap space with $L_1= \PP^N \times \AAA^1$ given by a non-toric presentation. However, by Remark \ref{remark : (d,-d) toric presentation}, for the degree $(d,-d)$ this space can be identified with the logarithmic quasimap space to $L_1$ with its toric presentation and with degree $d$.

\begin{proposition}\label{thm : 1 graph contribution} Let $\Delta:L_H\to L_H\times L_H$ be the diagonal embedding.
We have that  
    $$ [\cQ^{\log}_{0,2}(L_0,(d,0))^{\widetilde{H}}]^{\vir} = d \cdot j_{\widetilde{\Gamma},*}\left( c_{\widetilde{\Gamma}}^{*} \Delta^{!}\left([\cQ^{\log}_{0,(d,0)}(L_1|L_H,d)]^{\vir} \times [\cQ^{\log}_{0,(d,0)}(L_2|L_H,(0,d))^{H_{\infty}}]^{\vir}\right)\right),$$
\end{proposition}

\begin{proof}  
Applying the degeneration formula in \Cref{Th: degeneration} to $L$, we get
   \begin{equation*}
        [\cQ_{0,2}^{\log}(L_0,(d,0))^{\widetilde{H}}]^{\vir} = \sum_{\widetilde{\Gamma}} \frac{m_{\widetilde{\Gamma}}}{|\mathrm{Aut}(\widetilde{\Gamma})|} j_{\widetilde{\Gamma},*} c_{\widetilde{\Gamma}}^* \, \Delta^! \left(\prod_{V \in V(\widetilde{\Gamma})} [\cQ^{\log}_{0,\alpha_V}(L_{{r(V)}}|L_H,\beta_V)]^{\vir}\right)
    \end{equation*}
We now need to show that all terms except one vanish. 

This follows from combining \Cref{no several edges 1}, \Cref{no non-fibers}, \Cref{no ghosts}, \Cref{no purely p classes}. Finally, we see that for $\widetilde{\Gamma}$ the graph described above, we have $\frac{m_{\widetilde{\Gamma}}}{|\mathrm{Aut}(\widetilde{\Gamma})|} = d$ since there are no automorphisms and $m_{\widetilde{\Gamma}}$ can be given by the lowest common multiple of the weights.
\end{proof}

\begin{proposition}\label{Prop : qlocal P1 comp}
With the previous notations,
$$\ev_{1,*}\left([\cQ^{\log}_{(d,0)}(L_2|L_H,(0,d))^{H_{\infty}}]^{\vir}\right) = \frac{(-1)^{d+1}}{d}[H].$$
\end{proposition}

\begin{proof}
    As in \cite[Proposition 2.4]{vangarrel2019local}, it is enough to show via localisation that 
    \begin{equation}\label{eq:localisation_computation_qmaps}
        \deg\left(\ev_2^*(H_\infty) \cap e(\mathbf{R}^1\pi_\ast u^* \cO_{\PP^1}(-1)) \cap [\cQ^{\log}_{(d,0)}(\mathbb{P}^1|0,d)]^{\vir}\right) = \frac{(-1)^{d+1}}{d}.
    \end{equation}    

    By \Cref{lemma:localisation_only_one_graph}, there is a unique graph $\gamma$ contributing to the localisation computation. By the proof of \cite[Theorem 5.1]{Bryan_Pandharipande}, the same holds for stable maps. Since the fixed locus corresponding to $\gamma$ lies in the interior of both moduli spaces, 
    we have that
    \begin{align}\label{eq:localisation_maps=quasimaps}
        \deg\left(\ev_2^*(H_\infty) \cap e(\mathbf{R}^1\pi_\ast u^* \cO_{\PP^1}(-1)) \cap [\cQ^{\log}_{(d,0)}(\mathbb{P}^1|0,d)]^{\vir}\right) = \\
        =\deg\left(\ev_2^*(H_\infty) \cap e(\mathbf{R}^1\pi_\ast f^* \cO_{\PP^1}(-1)) \cap [\overline{\cM}^{\log}_{(d,0)}(\mathbb{P}^1|0,d)]^{\vir}\right) \nonumber
    \end{align}
    Therefore it is enough to show via localisation that
    \begin{equation}\label{eq:localisation_computation_maps}
        \deg\left(\ev_2^*(H_\infty) \cap e(\mathbf{R}^1\pi_\ast f^* \cO_{\PP^1}(-1)) \cap [\overline{\cM}^{\log}_{(d,0)}(\mathbb{P}^1|0,d)]^{\vir}\right) = \frac{(-1)^{d+1}}{d}.
    \end{equation}
    
    The fixed locus associated to $\gamma$ is isomorphic to $[\pt/\ZZ_d]$. The equivariant Euler class of the virtual normal bundle can be written as a product of three terms corresponding to the vertices, edges and flags of $\gamma$, see \cite{graber1999localization,graber2005relative}. 
    We choose a 1-dimensional subtorus with equivariant parameter $\hbar$ and a lift of the action to $\cO_{\PP^1}(-1)$ analogous to those in the proof of \cite[Theorem 5.1]{Bryan_Pandharipande} (with the roles of 0 and $\infty$ reversed).
    We compute that
    \[
        e_{\gamma}^F = -\frac{1}{\hbar^2}, \quad
        e_{\gamma}^v = \hbar^2, \quad
        e_{\gamma}^e = (-1)^d \frac{d!\, \hbar^d}{d^{d}}.
    \]
    Therefore the equivariant Euler class of the virtual normal bundle at the locus corresponding to $\gamma$ is
    \begin{equation}\label{eq:localisation_normal_bundle}
        e(N_\gamma^{\vir}) =  e_{\gamma}^F e_{\gamma}^v e_{\gamma}^e = (-1)^{d+1} \frac{d!\, \hbar^{d}}{d^{d}}.
    \end{equation}
    The contribution of the obstruction bundle is
    \begin{equation}\label{eq:localisation_obstruction}
        e(\mathbf{R}^1\pi_\ast f^* \cO_{\PP^1}(-1)) = \frac{(d-1)! \hbar^{d-1}}{d^{d-1}}.
    \end{equation}
    Multiplying together \eqref{eq:localisation_obstruction}, the inverse of \eqref{eq:localisation_normal_bundle}, a factor $1/d$ coming from the automorphism group of the fixed point corresponding to $\gamma$ and a factor $\hbar$ coming from the condition that the second mark is mapped to a hyperplane class, we get \eqref{eq:localisation_computation_maps}. We conclude by combining \eqref{eq:localisation_computation_qmaps}, \eqref{eq:localisation_maps=quasimaps} and \eqref{eq:localisation_computation_maps}.
\end{proof}

\begin{lemma}\label{lemma:localisation_only_one_graph}
    Consider the setting of \Cref{Prop : qlocal P1 comp}. In the localisation computation of the invariant
    \[
        \deg\left(\ev_2^*(H_\infty) \cap e(\mathbf{R}^1\pi_\ast f^* \cO_{\PP^1}(-1)) \cap [\cQ^{\log}_{(d,0)}(\PP^1 | 0, d)]^{\vir}\right)
    \]
    the unique graph whose contribution is non-zero is 
    \[
        \gamma = \xymatrix{
    \underset{V_1}{\overset{\{1\}}{\bullet}}
    \ar@{-}[r]^(-.1){}="a"^(1.1){}="b" \ar^d@{-} "a";"b"
    &\underset{V_2}{\overset{\{2\}}{\bullet}}
    }
    \]
    with $V_1$ mapped to 0 and $V_2$ mapped to $\infty$.
\end{lemma}

\begin{proof}
    Recall that, in the context of localisation, each vertex in a graph $\gamma$ corresponds to a connected components of the preimage of the fixed points on the target and each edge corresponds to a multiple cover of the target $\PP^1$. We decorate edges with the labels of the marked points and edges with the degree of the cover. 
    
    We would like to do the localisation calculation directly on $\cQ^{\log}_{0,(d,0)}(\PP^1|0,d)$. For a technical reason this is not possible, since the base of the perfect obstruction theory is not smooth. On the other hand one can use the procedure of ~\cite[Section 3.4]{ranganathan2022logarithmic} to produce a logarithmic modification 
    \begin{equation}\label{morphism : expand}
        \widetilde{\cQ^{\log}}_{0,(d,0)}(\PP^1|0,d) \rightarrow \cQ^{\log}_{0,(d,0)}(\PP^1|0,d)
    \end{equation}
    whose obstruction theory is defined relative to the space of logarithmic maps to expansions of $\sA$. 
    
    Consider a fixed locus $\cQ_\gamma$ in $\widetilde{\cQ^{\log}}_{0,(d,0)}(\PP^1|0,d)$. 
    
    Since elements in $\cQ^{\log}_{(d,0)}(\PP^1 |0, d)$ have irreducible source, the morphism \eqref{morphism : expand} forces quasimaps in $\cQ_\gamma$ to have irreducible source.
    
    The tangency condition $(d,0)$ ensures that the first marked point is mapped to 0. By the factor $\ev_2^*(H_\infty)$, the second marked point must be mapped to $\infty$. This is enough to conclude that the graph $\gamma$ in the statement is indeed the only contributing graph.   
\end{proof}

\begin{lemma}\label{no several edges 1}
    Let $\widetilde{\Gamma}$ be a decorated bipartite graph appearing in the degeneration formula. Then if $\widetilde{\Gamma}$ is not of the form, two vertices with a single edge connecting them, then $$[\cQ_{0,2}^{\log}(L_0,\widetilde{\Gamma})]^{\vir} = 0.$$  
\end{lemma}

\begin{proof}
    The exact same argument of ~\cite[Lemma 3.1]{vangarrel2019local} tells us that if $\widetilde{\Gamma}$ contains an $L_1$-vertex with at least two adjacent bounded edges then $[\cQ_{0,2}^{\log}(L_0,\widetilde{\Gamma})]^{\vir} = 0$. On the other hand, quasimap spaces in genus zero need to have at least $2$ markings by stability. Since we have two markings in total, one of which is forced to correspond to an $L_2$ vertex, we have that any graph which contributes must have two vertices and a single bounded edge between them with the markings distributed to either side.
\end{proof}

\begin{lemma}\label{no non-fibers}
     If $V$ is an $L_2$-vertex with $p_*\beta_V \neq 0$ then the class $[\cQ^{\log}_{0,\alpha_V}(L_2|L_H,\beta_V)]^{\vir}$ pushes forward to $0$ in $\cQ_{0,n_V + r_V}(H,p_{*}\beta_V)$.
\end{lemma}

\begin{proof}
By the computation of virtual dimensions in ~\cite[Lemma 5.1]{vangarrel2019local} we have that the virtual dimension of $\cQ^{\log}_{0,\alpha_V}(L_2|L_H,\beta_V)$ is strictly greater than the virtual dimension of $\cQ_{0, n_V + r_V}(H,p_* \beta_V)$ if $p_{*}\beta_V \neq 0$. On the other hand $\cQ_{0,n_V + r+V}(H,p_{*}\beta_V)$ is smooth of the expected dimension and so the result follows.
\end{proof}
\begin{notation}\label{gamma: one node}
Let $\widetilde{\Gamma}$ be a decorated bipartite graph with two vertices and a single edge $E$. We use the following (slightly abusive) notation for the universal curve and the irreducible components of the universal curve.
\[\xymatrix{\cC_{\widetilde{\Gamma}}\ar[d]^{\pi}\ar[r]^f& \cL_0\\ \cQ_{\widetilde{\Gamma}}}
\qquad
\xymatrix{\cC^i_{}\ar[d]^{\pi^i}\ar[r]^{f_i}&\cL_0\\ \cQ_{\widetilde{\Gamma}}}
\]
We denote by $P$ the section corresponding to the node.
\end{notation}

In the following, we want to exclude the contribution from ghost classes (see \Cref{def:ghost_class}). Thanks to \Cref{no several edges 1} and \Cref{no non-fibers}, we only need to rule out the following type of graphs. Below we use the basis $\{l,f\}$ of $H_2(W(\PP,H))$ introduced in \Cref{deformationPN}.

\begin{definition}
    Consider the double point degeneration $W(\PP^N,H)\to \AAA^1$. 
    Let $\widetilde{\Gamma}$ be a graph
    \[
        \xymatrix{
            \overset{V_1}{\bullet}
            \ar@{-}[r]^(-.1){}="a"^(1.1){}="b" \ar^E@{-} "a";"b"
            &\overset{V_2}{\bullet}
        }
    \] 
    with exactly one marking on each vertex, $V_1$ a $X_1$-vertex, $V_2$ an $X_2$-vertex, $\omega_E = e$ and with curve classes $\beta_{V_1} = (d,-e)$ and $\beta_{V_2}=(0,e)$ with $d\geq e \geq 0$. We say that $\widetilde{\Gamma}$ is a \textit{ghost graph} if $d\neq e$. Equivalently, $\widetilde{\Gamma}$ is a ghost graph if $\beta_{V_1}$ is a ghost class for the embedding $\PP^N\hookrightarrow W(\PP^N,H)$ through the special fibre, in the sense of \Cref{def:ghost_class}.
\end{definition}

\begin{lemma}\label{p h smooth}
Let $\widetilde{\Gamma}$ be a ghost graph.
The moduli space $\cQ_{0,2}^{\log}(W_0(\PP^N,H),\widetilde{\Gamma})$ is smooth over $\frPic^{\log,s}_{0,2,\underline{d}}(\mathscr{A}_0,\Gamma)$. 
\end{lemma}

\begin{proof} 
    We show that $\cQ_{0,2}^{\log}(W_0(\PP^N,H),\widetilde{\Gamma})$ is unobstructed over $\frPic^{\log,s}_{0,2,\underline{d}}(\mathscr{A}_0,\Gamma)$.
    Using the description of the toric divisors of $W(\PP^N,H)$ in \Cref{deformationPN}, we see that
    \begin{equation}\label{eq:intersection_numbers}
    \begin{aligned}[c]
    \widetilde{H}_1\cdot(d,-e)&=d\\
    \widetilde{H}\cdot(d,-e)&=d-e
    \end{aligned}
    \qquad\qquad
    \begin{aligned}[c]
    \widetilde{H}_1\cdot(0,e)&=0\\
    \widetilde{H}\cdot(0,e)&=e.
    \end{aligned}
    \end{equation}
    
    We consider the normalization sequence of $\cC$, the universal curve over $\cQ_{0,2}^{\log}(W_0(\PP^N,H),\widetilde{\Gamma})$,
    \[
        0\to\cO_{\cC}\to \cO_{\cC_1}\oplus \cO_{\cC_2}\to \cO_{P}\to 0.
    \]
    
    Given the obstruction theory in \Cref{prop:POT_W_relative_Pic}, we take the sheaf
    \begin{equation}\label{eq:def_sheaf_curly_E}
        \cE(W) \coloneqq \oplus_{i=1}^N\cO_{\cC}(\widetilde{H}_i)\oplus \cO_{\cC}(\widetilde{H})\oplus \cO_{\cC}.
    \end{equation}
    on $\cC$ and the sheaves
    \[
        \cE_{X_i|H} \coloneqq \oplus_{\rho \neq \rho_{0} \in \Sigma_{X_i}(1)} \cO_{\cC_i}(D_{\rho})\oplus \cO_{\cC_i}
    \]
    on $\cC_i$ for $i=1,2$, with $D_{\rho_0} = H$ in both cases.

    Tensoring the normalization sequence of $\cO_{\cC}$ with $\cE(W)$ and pushing forward to $\cQ_{0,2}^{\log}(W_0(\PP^N,H),\widetilde{\Gamma})$, we get a long exact sequence
    \begin{align*}
    0\to &\mathbf{R}^0\pi_* \cE(W) \to \mathbf{R}^0\pi^1_*\cE_{X_1|H} \oplus \mathbf{R}^0\pi^2_*\cE_{X_2|H} \stackrel{\phi}{\rightarrow} \mathbf{R}^0\pi_*\cE|_P\to\\
    \to &\mathbf{R}^1\pi_* \cE(W)\to \mathbf{R}^1\pi^1_*\cE_{X_1|H} \oplus \mathbf{R}^1\pi^2_*\cE_{X_2|H}\to 0.
    \end{align*}
    Since $\phi$ is surjective and since $\mathbf{R}^1\pi^1_*\cE_{X_1|H} \oplus \mathbf{R}^1\pi^2_*\cE_{X_2|H} = 0$ by \eqref{eq:intersection_numbers}, we conclude that $\mathbf{R}^1\pi_* \cE(W)=0$. This proves the claim.
\end{proof}

\begin{corollary}\label{ghost w smooth} 
Let $\tilde{\Gamma}$ be a ghost graph.
The moduli space $\cQ_{0,2}^{\log}(L_0,\tilde{\Gamma})$ is smooth over $\frPic^{\log,s}_{0,2,\underline{d}}(\mathscr{A}'_0,\Gamma)$.
\end{corollary}

\begin{proof} Since $\widetilde{\cD}$ is effective, we have that maps from (nodal) curves to $\cL$ factor through the zero section. This shows that 
$\cQ_{0,2}^{\log}(W(\PP^N,H),\tilde{\Gamma})\simeq \cQ_{0,2}^{\log}(L/\AAA^1,\tilde{\Gamma})$.
\Cref{p h smooth} implies the claim.
\end{proof}

\begin{proposition}\label{no ghosts}
    Let $\tilde{\Gamma}$ be a ghost graph.
    Then
    \[[\cQ_{0,2}^{\log}(L_0,\tilde{\Gamma})]^\vir=0\]
\end{proposition}

\begin{proof} 
In this proof we write $W$ for $W(\PP^N,H)$, as there is no risk of confusion. 

Let $\cC$ denote the universal curve over $\cQ_{0,2}^{\log}(L_0,\tilde{\Gamma}) \simeq \cQ_{0,2}^{\log}(W_0,\tilde{\Gamma})$. By \Cref{prop:POT_W_relative_Pic}, the obstruction theory on $\cQ_{0,2}^{\log}(W_0,\tilde{\Gamma})$ is determined by the sheaf $\cE(W)$ on $\cC$, defined in \eqref{eq:def_sheaf_curly_E}, while the one on $\cQ_{0,2}^{\log}(L_0,\tilde{\Gamma})$ is controlled by 
\begin{equation}\label{toric div for l}
    \cE(L) = \cE(W) \oplus \cO_{\cC}(Z).
\end{equation}
Tensoring the normalization sequence of $\cO_\cC$ with $\cE(L)$, we get the following long exact sequence on $\cQ_{0,2}^{\log}(L_0,\tilde{\Gamma})$
\begin{align}\label{long obs seq}
    0\to &\mathbf{R}^0\pi_* \cE(L) \to \mathbf{R}^0\pi^1_*\cE_{L_1|H} \oplus \mathbf{R}^0\pi^2_*\cE_{L_2|H} \stackrel{\phi}{\rightarrow} \mathbf{R}^0\pi_*\cE(L)|_P\to\\
    \to &\mathbf{R}^1\pi_* \cE(L)\to \mathbf{R}^1\pi^1_*\cE_{L_1|H} \oplus \mathbf{R}^1\pi^2_*\cE_{L_2|H}\to 0\nonumber
\end{align}
with 
\begin{equation}\label{eq:comparison_E_Li_E_Xi}
    \cE_{L_i|H} = \cE_{X_i|H} \oplus \cO_{\cC_i}(Z).
\end{equation}

By \Cref{ghost w smooth}, $\cQ_{0,2}^{\log}(W_0,\tilde{\Gamma}) \simeq  \cQ_{0,2}^{\log}(L_0,\tilde{\Gamma})$ is smooth over $\frPic^{\log,s}_{0,2,\underline{d}}(\mathscr{A}'_0)$ and it has $\mathbf{R}^1\pi_*\cE(L)$ as an obstruction bundle. This shows that
\begin{equation}\label{cap with euler}
    [\cQ_{0,2}^{\log}(L_0,\tilde{\Gamma})]^\vir=
    e(\mathbf{R}^1\pi_* \cE(L)) \cap 
    [\cQ_{0,2}^{\log}(W_0,\tilde{\Gamma})]^\vir.
\end{equation}
In the following we show that the Euler class of $\mathbf{R}^1\pi_* \cE$ is zero.

We first show that we have an exact sequence
\begin{equation} \label{r0 obs seq}
    \mathbf{R}^0\pi^1_*\cE_{L_1|H} \oplus \mathbf{R}^0\pi^2_*\cE_{L_2|H} \stackrel{\phi}{\rightarrow} 
    \mathbf{R}^0\pi_*\cE(L)|_P \to
    \mathbf{R}^0\pi_*\cO_{\cC}(Z)|_P \to 0.
\end{equation}
For that, we use \eqref{toric div for l} and \eqref{eq:comparison_E_Li_E_Xi} to re-write the morphism $\phi$ as follows
\[
    (\mathbf{R}^0\pi^1_*\cE_{X_1|H} \oplus \mathbf{R}^0\pi^2_*\cE_{X_2|H}) \oplus (\mathbf{R}^0\pi^1_*\cO_{\cC_1}(Z) \oplus \mathbf{R}^0\pi^2_*\cO_{\cC_2}(Z))\stackrel{\phi}{\rightarrow} 
    (\mathbf{R}^0\pi_*\cE(W)|_P) \oplus \mathbf{R}^0\pi_*\cO_{\cC}(Z)|_P.
\]
We have that
\[
    \mathbf{R}^0\pi^1_*\cE_{X_1|H} \oplus \mathbf{R}^0\pi^2_*\cE_{X_2|H} \to
    \mathbf{R}^0\pi_*\cE(W)|_P
\]
is surjective and by \Cref{deformationPN} we have that
\begin{equation}
(d,-e)\cdot [Z]=e-d<0 \qquad \text{and}\qquad (0,e)\cdot [Z]=-e<0,
\end{equation}
which shows that $\mathbf{R}^0\pi^1_*\cO_{\cC_1}(Z) \oplus \mathbf{R}^0\pi^2_*\cO_{\cC_2}(Z)=0$. This shows that the cokernel of $\phi$ is $\mathbf{R}^0\pi_*\cO_{\cC}(Z)|_P$, which proves the exactness of \eqref{r0 obs seq}.

Combining \eqref{r0 obs seq} and \eqref{long obs seq}, we get a short exact sequence
\begin{equation}\label{r1 obs seq}
0\to \mathbf{R}^0\pi_*\cO_{\cC}(Z)|_P\to 
\mathbf{R}^1\pi_* \cE(L)\to 
\mathbf{R}^1\pi^1_*\cE_{L_1|H} \oplus \mathbf{R}^1\pi^2_*\cE_{L_2|H}\to 0,
\end{equation}
which implies that
\begin{align}\label{proof euler}
e(\mathbf{R}^1\pi_*\cE(L)) = 
e(\mathbf{R}^0\pi_*\cO_{\cC}(Z)|_P) \cdot
e(\mathbf{R}^1\pi^1_*\cE_{L_1|H} \oplus \mathbf{R}^1\pi^2_*\cE_{L_2|H}).
\end{align}
Therefore, it is enough to show that the Euler class of $\mathbf{R}^0\pi_*\cO_{\cC}(Z)|_P$ is zero. 

For this, we notice that $[Z]=-p^*[\widetilde{H}]$ with $p\colon L\to W$ the natural projection. It follows that $e(\mathbf{R}^0\pi_*\cO_{\cC}(Z)|_P)=-\ev_P^*(p^*\widetilde{H})$. As $p^*\widetilde{H}$ has a universal section which does not vanish at $P$, we have that $(p^*\widetilde{H})|_P$ has a nowhere vanishing section and thus $\ev_P^*(p^*\widetilde{H})$ is a trivial line bundle. This shows that 
\begin{equation}\label{zero euler}
    e(\mathbf{R}^0\pi_*\cO_{\cC}(Z)|_P)=0.
\end{equation}
The equalities \eqref{cap with euler}, \eqref{proof euler} and \eqref{zero euler} conclude the proof.
\end{proof}

\begin{lemma}\label{no purely p classes}
   Graphs contained in $L_1$ do not contribute to the degeneration formula. 
\end{lemma}
\begin{proof}
Intersection of these with $\ev^*D_{\infty}$ is zero.
\end{proof}

\begin{proof}[Proof of \Cref{qlocallogtoric}]
By \Cref{thm : 1 graph contribution}, together with commutativity of diagram \eqref{prod-qmaps}, we have that
    \begin{align}
        &\quad \, P_{*}[\cQ_{0,2}^{\log}(L_0,(d,0))^{\widetilde{H}}]^{\vir}
        \nonumber\\
        &= d \cdot F_{*}\mathrm{pr}_{1,*}c_{\widetilde{\Gamma},*}c_{\widetilde{\Gamma}}^* \Delta^{!}\left([\cQ^{\log}_{0,(d,0)}(L_1|L_H,d)]^{\vir} \times [\cQ^{\log}_{0,(d,0)}(L_2|L_H,(0,d))^{H_{\infty}}]^{\vir}\right)\nonumber \\
        &= d \cdot F_{*}\mathrm{pr}_{1,*}\Delta^{!}\left([\cQ^{\log}_{0,(d,0)}(L_1|L_H,d)]^{\vir} \times [\cQ^{\log}_{0,(d,0)}(L_2|L_H,(0,d))^{H_{\infty}}]^{\vir}\right) \nonumber.
    \end{align}

On the other hand, $[\cQ^{\log}_{0,(d,0)}(L_1|L_H,d)]^{\vir}$ Gysin-restricts to $[\cQ^{\log}_{0,(d,0)}(\PP^N|H,d)]^{\vir}$ when confining the evaluation map to be in $H \times \{0\}$. Combining this with proposition \Cref{Prop : qlocal P1 comp} tells us that
 \begin{align*}\label{reformulation}
\mathrm{pr}_{1,*}\Delta^{!}\left([\cQ^{\log}_{0,(d,0)}(L_1|L_H,d)]^{\vir} \times [\cQ^{\log}_{0,(d,0)}(L_2|L_H,(0,d))^{H_{\infty}}]^{\vir}\right) =\frac{(-1)^{d+1} }{d}[\cQ^{\log}_{0,(d,0)}(\PP^N|H,d)]^{\vir}.
\end{align*}

Finally, combining this with \Cref{Lemma : local pushdown} tells us that 
\begin{equation*}
P_{*}[\cQ^{\log}_{0,2}(L_0,(d,0))^{\widetilde{H}}]^{\vir} = \ev_2^*H \cap [\cQ_{0,2}(\cO_{\PP^N}(-H),d)]^{\vir} = (-1)^{d+1} F_*[\cQ^{\log}_{0,(d,0)}(\PP^N|H,d)]^{\vir},
\end{equation*}
which completes the proof. 
\end{proof}

\subsection{Simple normal crossings divisors}
In \cite{vangarrel2019local} the authors conjecture a generalisation of the local/logarithmic correspondence to the simple normal crossings situation. In the strong form of the correspondence, this would be

\begin{equation*}
    [\overline{\cM}_{0,((d_1,\dots,d_r))}^{\log}(X|D,\beta)]^{\vir} = \prod_{i=1}^r(-1)^{d_i+1} \cdot \ev_i^* D_i \cap [\overline{\cM}_{0,r}(\oplus_{i=1}^r\cO_{X}(-D_i),\beta)]^{\vir}
\end{equation*}
where the left hand side indicates $r$-marked curves in $X$ with maximal tangency to $D_i$ at the $i^{\mathrm{th}}$ marking. The classes are viewed as living in $\overline{\cM}_{0,r}(X,\beta)$ but in this section we will omit the pushforward on the left-hand side. Although this generalisation often holds numerically, \cite{bousseau2022log,bousseau2024stable}, 
at the level of virtual classes it is not true in complete generality, due to counterexamples in \cite{nabijou2022gromov}. However, in ~\cite[Remark 5.4]{nabijou2022gromov} the authors note that the difference appears when there are components in the moduli space containing rational tails and so may not be present in the quasimap setting.

It will turn out that an easy corollary of the local/logarithmic correspondence for smooth divisors and properties of logarithmic quasimaps is the following. 
\begin{theorem}\label{snclocallog}

 Let $X$ be a GIT quotient $W \GIT G$ and let $D \subset X$ be a \emph{very ample} divisor satisfying Assumptions \ref{assumptions:absolute_quasimaps} and \ref{assumptions:log_quasimaps}.
 Suppose that $D=D_1 + D_2$ has two components, which are very ample. Let $\beta$ be a curve class on $X$, let $d_1 = D_1 \cdot \beta$ and let $d_2 = D_2 \cdot \beta$. Then we have the following equality of virtual classes
$$[\cQ_{0,((d_1,d_2))}^{\log}(X|D_1 + D_2,\beta)]^{\vir} = (-1)^{d_1+d_2}\prod_{i=1}^2 \ev_i^* D_i \cap [\cQ_{0,2}(\oplus_{i=1}^2\cO_{X}(-D_i),\beta)]^{\vir}.$$
\end{theorem}
In fact there is a much more general conjecture one could make for the local/logarithmic correspondence for simple normal crossings divisors \cite{tseng2023mirror}. The conjecture of \cite{vangarrel2019local} is one extreme. The other extreme, we call the \emph{corner} theory. On the logarithmic side the geometry is of curves with a single relative marking with maximal tangency to \emph{all} divisor components simultaneously. Using exactly the same techniques we prove that this generalisation holds for quasimap theory \emph{in any rank}.
\begin{figure}[ht]
    \centering
    \includegraphics[width=10.5cm]{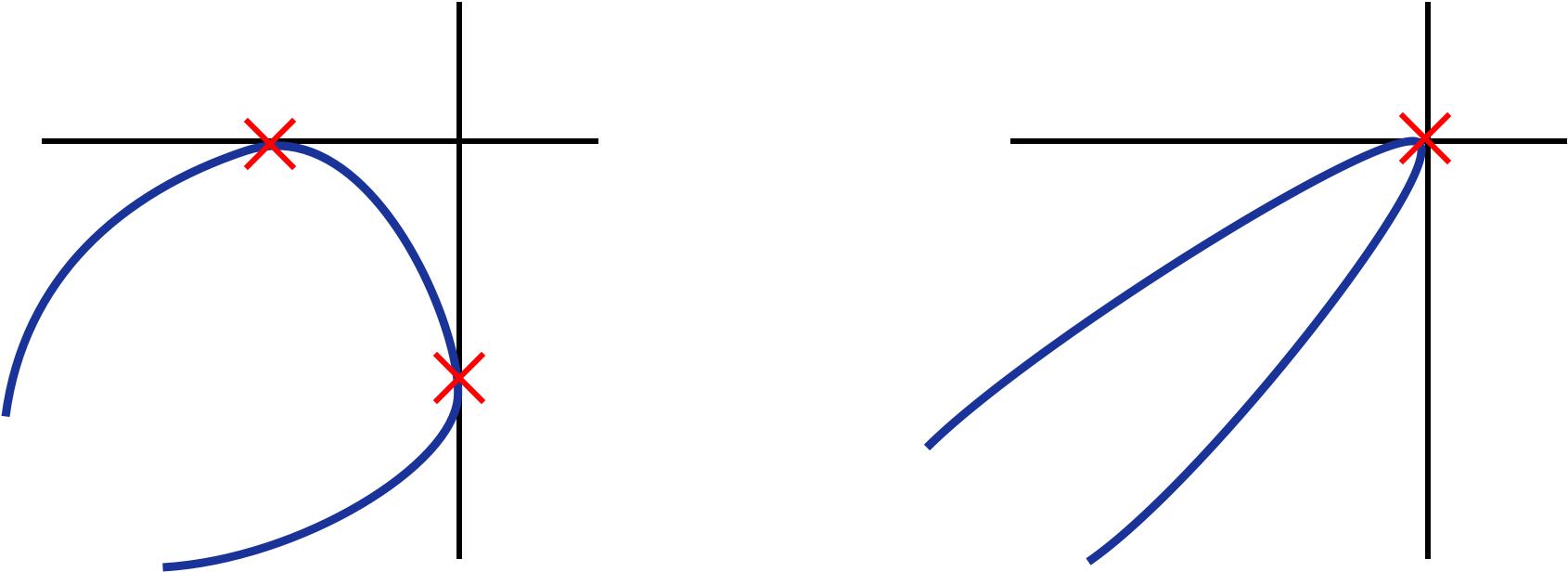}
    \caption{ The two s.n.c. maximal tangency geometries. Left-hand side depicts \cite{vangarrel2019local} version and right-hand side depicts the \emph{corner theory}.}\label{fig:cornervsvGGR}
\end{figure}

\begin{theorem}\label{qloglocalcorner}
   Let $X$ be a GIT quotient $W \GIT G$ and let $D = D_1 + \dots + D_r \subset X$ be a \emph{very ample} divisor satisfying Assumptions \ref{assumptions:absolute_quasimaps} and \ref{assumptions:log_quasimaps}
   with $D_1 \cap \dots \cap D_r \neq \emptyset$. Then 
\begin{equation*}
     [\cQ_{0,(\underline{d},0)}^{\log}(X|D,\beta)]^{\vir} = \prod_{i=1}^r(-1)^{d_i+1} \cdot \ev_1^* D_i \cap [\cQ_{0,2}(\oplus_{i=1}^r\cO_{X}(-D_i),\beta)]^{\vir}.
\end{equation*}
\end{theorem}

We start with \Cref{snclocallog}. We first prove the equality in the following special case.
\begin{theorem}\label{thm:cartesian_log_qmap_snc}
    Let $X= \PP^{N_1} \times \PP^{N_2}$. Let $H_1' = H_1 \times \PP^{N_2}$ for $H_1$ a coordinate hyperplane in $\PP^{N_1}$ and $H_2' = \PP^{N_1} \times H_2$ for $H_2$ a coordinate hyperplane in $\PP^{N_2}$. Then the following diagram is Cartesian in the category of ordinary stacks.
    \begin{equation*}
         \begin{tikzcd}
            {\cQ_{0,((d_1,d_2))}^{\log}(\PP^{N_1} \times \PP^{N_2}|H_1' + H_2',(d_1,d_2))} \arrow[d] \arrow[r] & {\cQ_{0,(d_1,0)}^{\log}(\PP^{N_1} \times \PP^{N_2}|H_1',(d_1,d_2))} \arrow[d] \\
            {\cQ_{0,(0,d_2)}^{\log}(\PP^{N_1} \times \PP^{N_2}|H_2',(d_1,d_2))} \arrow[r]                   & {\cQ_{0,2}(\PP^{N_1} \times \PP^{N_2},(d_1,d_2))}                         
        \end{tikzcd}
    \end{equation*}
\end{theorem}

\begin{proof}
By stability, the boundary divisors of $\cQ_{0,(d_1,0)}^{\log}(\PP^{N_1} \times \PP^{N_2}|H_1',(d_1,d_2))$ consist of quasimaps from curves with at most two components ~\cite[Example 3.1]{shafi2024logarithmic}. As a consequence the right-hand vertical morphism (or indeed the bottom horizontal morphism) is such that the morphism on tropicalisations is combinatorially flat. Therefore the fine and saturated fibre product coincides with ordinary fibre product ~\cite[Proposition 1.7.3]{maulik2024logarithmic}, but the fine and saturated fibre product is just $\cQ_{0,((d_1,d_2))}^{\log}(\PP^{N_1} \times \PP^{N_2}|H_1' + H_2',(d_1,d_2))$, ~\cite[Lemma 2.9]{shafi2024logarithmic}.
\end{proof}

\begin{proposition}\label{sncloglocalunobstructed}
    With the notation of \Cref{thm:cartesian_log_qmap_snc}, we have the following equality of (virtual) classes
    $$[\cQ_{0,((d_1,d_2))}^{\log}(\PP^{N_1} \times \PP^{N_2}|H_1' + H_2',(d_1,d_2))]^{\vir} = (-1)^{d_1 + d_2}\prod_{i=1}^2 \ev_i^* H_i' \,\cap \, [\cQ_{0,2}(\oplus_{i=1}^2\cO_{\PP^{N_1} \times \PP^{N_2}}(-H_i'),(d_1,d_2))]^{\vir}$$
\end{proposition}

\begin{proof}
    The local virtual class is $$[\cQ_{0,2}(\oplus_{i=1}^2\cO_{\PP^{N_1} \times \PP^{N_2}}(-H_i'),(d_1,d_2))]^{\vir} =  e(\oplus_{i=1}^2 \mathbf{R}^1\pi_*(\mathscr{L}_i^{\vee}))\cap[\cQ_{0,2}(\PP^{N_1} \times \PP^{N_2},(d_1,d_2)]$$
where $\mathscr{L}_i$ denotes the line bundle over the universal curve of $\cQ_{0,2}(\PP^{N_1} \times \PP^{N_2},(d_1,d_2))$ corresponding to $\cO_{\PP^{N_1} \times \PP^{N_2}}(H_i')$. But because of the Whitney sum formula, and the fact that pushforward commutes with direct sum, we have that 
\begin{align*}
    &e(\oplus_{i=1}^2\mathbf{R}^1\pi_*(\mathscr{L}_i^{\vee}))\cap[\cQ_{0,2}(\PP^{N_1} \times \PP^{N_2},(d_1,d_2))] \\
    = &\prod_{i=1}^2 e(\mathbf{R}^1 \pi_* \mathscr{L}_i^{\vee}) \cap [\cQ_{0,2}(\PP^{N_1} \times \PP^{N_2},(d_1,d_2))] \\
    = &\prod_{i=1}^2 [\cQ_{0,2}(\cO_{\PP^{N_1} \times \PP^{N_2}}(-H_i'),(d_1,d_2))]^{\vir}.
\end{align*}
By \Cref{qlocallog} 
\begin{align*}\hspace{-9em}
    &[\cQ_{0,(d_1,0)}^{\log}(\PP^{N_1} \times \PP^{N_2}|H'_1,(d_1,d_2))] \cdot [\cQ_{0,(0,d_2)}^{\log}(\PP^{N_1} \times \PP^{N_2}|H_2',(d_1,d_2))]\\
    &= (-1)^{(d_1+d_2)}\prod_{i=1}^2 \ev_i^* H_i \cdot [\cQ_{0,2}(\oplus_{i=1}^2\cO_{\PP^{N_1} \times \PP^{N_2}}(-H_i),(d_1,d_2))]^{\vir}
\end{align*}
Therefore it suffices to show that 
\begin{align*}\hspace{-3em}
    &[\cQ_{0,((d_1,d_2))}^{\log}(\PP^{N_1} \times \PP^{N_2}|H_1' + H_2',(d_1,d_2))]^{\vir}\\
    &= [\cQ_{0,(d_1,0)}^{\log}(\PP^{N_1} \times \PP^{N_2}|H_1',(d_1,d_2))] \cdot [\cQ_{0,(0,d_2)}^{\log}(\PP^{N_1} \times \PP^{N_2}|H_2',(d_1,d_2))].
\end{align*}
This follows from \Cref{thm:cartesian_log_qmap_snc}.
\end{proof}

\begin{remark}
    The natural generalisation of this result for rank greater than $2$ will fail. The introduction of additional markings will allow for boundary divisors in the logarithmic quasimap space where the source curve contains more than two components and so the fine and saturated product and the ordinary fibre product may differ \cite[Remark 3.7.6]{shafi2022enumerative}.
\end{remark}

\begin{lemma}\label{Lemma : POTpullback}
   Let $X = Z \GIT G$ be a GIT quotient and $D$ be a simple normal crossings divisor on $X$ satisfying Assumptions \ref{assumptions:absolute_quasimaps} and \ref{assumptions:log_quasimaps}.
   In addition suppose that $D=D_1 + D_2$ has two components, which are very ample. Let $j: X \rightarrow \PP = \PP^{N_1} \times \PP^{N_2}$ be the induced morphism with $j^{-1}(H_1 \times \PP^{N_2})= D_1$ and $j^{-1}(\PP^{N_1} \times H_2) = D_2$ and assume $H_1,H_2$ are coordinate hyperplanes. Then there is an induced morphism $$i : \cQ^{\log}_{0,\alpha}(X|D_1 + D_2, \beta) \rightarrow \cQ_{0,\alpha}^{\log}(\PP|H_1' + H_2',(d_1,d_2))$$ where $H_1' = H_1 \times \PP^{N_2}$, $H_2' = \PP^{N_1} \times H_2$ and $i_{*}\beta = (d_1,d_2)$. Moreover, this morphism admits a perfect obstruction theory and the virtual pullback of the class $[\cQ_{0,\alpha}^{\log}(\PP|H_1' + H_2',(d_1,d_2))]$ coincides with the virtual class $[\cQ^{\log}_{0,\alpha}(X|D_1 + D_2, \beta)]^{\vir}$ as in ~\cite[Theorem 14]{shafi2024logarithmic}. 
\end{lemma}

\begin{proof}
    This proof is similar to \cite[5.10]{shafi2024logarithmic}.The morphism $j$ morphism induces a morphism of quotient stacks, which we also denote $j$, $$j: [Z/G] \rightarrow [\AAA^{N_1+N_2 + 2}/\GG^2_m]$$ such that $j^{-1}(\bar{H_1'}) = \bar{D_1}$ and $j^{-1}(\bar{H_2'})$. The divisors $\bar{D_1} + \bar{D_2},\bar{H_1'} + \bar{H_2'}$ define logarithmic structures via morphisms to $[\AAA^2/\GG^2_m]$ such that there is a commutative diagram
	\begin{equation*}
		\begin{tikzcd}
		{[Z/G]} \arrow[r, "j"] \arrow[rr, "\bar{D_1} + \bar{D_2}", bend left] & {[\AAA^{N_1+N_2 + 2}/\GG^2_m]} \arrow[r, "\bar{H_1'
        } + \bar{H_2'}"] & \agm.
		\end{tikzcd}
	\end{equation*}

Taking the induced distinguished triangle (on tangent complexes) gives
    \begin{equation*}
\TT_{j} \rightarrow \TT^{\log}_{[Z/G]} \rightarrow \mathbf{L}j^*\TT^{\log}_{[\AAA^{N_1+N_2 + 2}/\GG^2_m]} \rightarrow \TT_{j}[1].
\end{equation*}

Next, consider the diagram
\begin{equation}\label{diagram0}
	\begin{tikzcd}
	\cC_{X} \arrow[d, "\pi"] \arrow[r] \arrow[rr, "u", bend left] & \cC_{\PP^N} \arrow[d, "\pi_{\PP}"] \arrow[rr, "u_{\PP}", bend right] & {[Z/G]} \arrow[r, "j"] & {[\AAA^{N_1+N_2 + 2}/\GG^2_m]} \\
	\cQ^{\log}_{0,\alpha}(X|D,\beta) \arrow[r, "i"]                                              & {\cQ^{\log}_{0,\alpha}(\PP|H_1' + H_2',(d_1,d_2))}.                                          &                                   &                     
	\end{tikzcd}
\end{equation}
If we apply $\mathbf{R}\pi_{*} \circ \mathbf{L}u^{*}$ and dualise, we get a distinguished triangle in $\cQ^{\log}_{0,\alpha}(X|D,\beta)$
\begin{equation*}\hspace{-3em}
\mathbf{L}i^*(\mathbf{R}(\pi_{\PP})_{*}\mathbf{L}u_{\PP}^*\TT^{\log}_{[\AAA^{N_1+N_2 + 2}/\GG^2_m]})^{\vee} \rightarrow 	(\mathbf{R}\pi_{*}\mathbf{L}u^*\TT^{\log}_{[Z/G]})^{\vee} \rightarrow (\mathbf{R}\pi_{*}\mathbf{L}u^*\TT_{j})^{\vee} \rightarrow \mathbf{L}i^*(\mathbf{R}(\pi_{\PP})_{*}Lu_{\PP}^*\TT^{\log}_{[\AAA^{N_1 + N_2 +2}/\GG_m^2]})^{\vee}[1].
\end{equation*}
Note that the first two complexes in the triangle define the obstruction theories for $\cQ^{\log}_{0,\alpha}(X|D,\beta)$ and $\cQ^{\log}_{0,\alpha}(\PP|H_1' + H_2',(d_1,d_2))$.

We also have the commutative diagram
	\begin{equation*}
\begin{tikzcd}
{\cQ^{\log}_{0,\alpha}(X|D,\beta)} \arrow[r, "i"] \arrow[rr, bend left] & {\cQ^{\log}_{0,\alpha}(\PP|H_1' + H_2',(d_1,d_2))} \arrow[r] & \frM^{\log}_{0,\alpha}([\AAA^2/\GG_m^2])
\end{tikzcd}
\end{equation*}
inducing a distinguished triangle $$\mathbf{L}i^*\LL_{\cQ^{\log}(\PP)/\frM^{\log}} \rightarrow \LL_{\cQ^{\log}(X)/\frM^{\log}} \rightarrow \LL_{i} \rightarrow \mathbf{L}i^*\LL_{\cQ^{\log}(\PP)/\frM^{\log}}[1]$$
where $\cQ^{\log}(X) = \cQ^{\log}_{0,\alpha}(X|D,\beta), \, \cQ^{\log}(\PP) = \cQ^{\log}_{0,\alpha}(\PP|H_1' + H_2',(d_1,d_2))$ and $\frM^{\log} = \frM^{\log}_{0,\alpha}([\AAA^2/\GG_m^2])$. Putting these together we get
\begin{equation*}
	\begin{tikzcd}
{\mathbf{L}i^*(\mathbf{R}(\pi_{\PP})_{*}\mathbf{L}u_{\PP}^*\TT^{\log}_{[\AAA^{N_1 + N_2 + 2}/\GG^2_m]})^{\vee}} \arrow[r] \arrow[d] & {(\mathbf{R}\pi_{*}\mathbf{L}u^*\TT^{\log}_{[Z/G]})^{\vee}} \arrow[r] \arrow[d] & (\mathbf{R}\pi_{*}\mathbf{L}u^*\TT_{j})^{\vee} \arrow[r, "{[1]}"] \arrow[d] & {} \\
\mathbf{L}i^*\LL_{\cQ^{\log}(\PP)/\frM^{\log}} \arrow[r]                                                                & \LL_{\cQ^{\log}(X)/\frM^{\log}} \arrow[r]                                                  & \LL_{i} \arrow[r, "{[1]}"]                                                  & {}.
\end{tikzcd}
\end{equation*}

The first two (and last) vertical arrows come from the perfect obstruction theory from \cite{shafi2024logarithmic}. By ~\cite[Construction 3.13]{manolache12virtual} and ~\cite[Remark 3.15]{manolache12virtual} it follows that $(\mathbf{R}\pi_{*}\mathbf{L}u^*\TT_{i})^{\vee}$ defines a perfect obstruction theory. By ~\cite[Theorem 4.8]{manolache12virtual} we have that $i^{!}[\cQ^{\log}_{0,\alpha}(\PP|H_1' + H_2',(d_1,d_2))] = [\cQ^{\log}_{0,\alpha}(X|D,\beta)]^{\vir}$. 
\end{proof}
\begin{remark}
    One can show that this virtual pullback coincides with the pullback coming form a perfect obstruction theory of the morphism $\cQ_{0,2}(X,\beta) \rightarrow \cQ_{0,2}(\PP,(d_1,d_2))$.
\end{remark}

\begin{proof}[Proof of \Cref{snclocallog}]
    \begin{equation}
        \begin{tikzcd}
{\cQ_{0,((d_1,d_2))}^{\log}(X|D_1 + D_2,\beta)} \arrow[d,"p"] \arrow[r,"i"] & {\cQ_{0,((d_1,d_2))}^{\log}(\PP|H_1' + H_2',(d_1,d_2))} \arrow[d,"p'"] \\
{\cQ_{0,2}(X,\beta)} \arrow[r,"i'"]                        & {\cQ_{0,2}(\PP,(d_1,d_2))}                       
\end{tikzcd}
    \end{equation}
Let $\mathscr{L}_i$ (resp. $\mathscr{L}_{D_i}$) denote the universal line bundle on the universal curve $\cQ_{0,2}(\PP,\underline{d})$ (resp. $\cQ_{0,2}(X,\beta)$) corresponding to $\cO_{\PP}(H_i')$ (resp. $\cO_{X}(D_i)$).
\begin{align*}
    &p_*[\cQ_{0,((d_1,d_2))}^{\log}(X|D_1 + D_2,\beta)]^{\vir} \\
    &= p_* (i')^{!}[\cQ_{0,((d_1,d_2))}^{\log}(\PP|H_1' + H_2',(d_1,d_2))]  \\
    &= {i'}^{!}{p'}_{*}[\cQ_{0,((d_1,d_2))}^{\log}(\PP|H_1' + H_2',(d_1,d_2))] \\
    &= {i'}^{!}(-1)^{d_1 + d_2}\prod_{i=1}^2 \ev_i^* H_i' \cap [\cQ_{0,2}(\oplus_{i=1}^2\cO_{\PP^{N_1} \times \PP^{N_2}}(-H_i'),(d_1,d_2))]^{\vir} \\
    &= (-1)^{d_1 + d_2}\prod_{i=1}^2 \ev_i^* D_i \cap {i'}^{!}[\cQ_{0,2}(\oplus_{i=1}^2\cO_{\PP^{N_1} \times \PP^{N_2}}(-H_i'),(d_1,d_2))]^{\vir} \\
    &=  (-1)^{d_1 + d_2}\prod_{i=1}^2 \ev_i^* D_i \cap {i'}^!e(\oplus_{i=1}^2\textbf{R}^1\pi_*\mathscr{L}_i^{\vee}) \cap [\cQ_{0,2}(\PP,(d_1,d_2))] \\
    &= (-1)^{d_1 + d_2}\prod_{i=1}^2 \ev_i^* D_i \cap e(\oplus_{i=1}^2\textbf{R}^1\pi_*\mathscr{L}_{D_i}^{\vee}) \cap [\cQ_{0,2}(X,\beta)] \\
    &= (-1)^{d_1 + d_2}\prod_{i=1}^2 \ev_i^* D_i \cap [\cQ_{0,2}(\oplus_{i=1}^2\cO_{X}(-D_i),\beta)]^{\vir}.\qedhere
\end{align*}
\end{proof}

Finally, we move on to proving \Cref{qloglocalcorner}. 
With the same argument of \Cref{thm:cartesian_log_qmap_snc}, we get the following

\begin{theorem}\label{thm:cartesian_log_qmap_snc_r}
    Let $\PP = \prod_{i=1}^{r}\PP^{N_i}$ and let $H_i' = H_i \times \prod_{j \neq i} \PP^{N_j}$ for $H_i$ a coordinate hyperplane in $\PP^{N_i}$ and $H' = \sum_{i=1}^r H'_i$. There is a Cartesian diagram in the category of ordinary stacks
    \begin{equation*}
    \begin{tikzcd}
\cQ^{\log}_{0,(\underline{d},0)}(\PP|H',d) \arrow[d] \arrow[r]            & {\prod_{i=1}^r \cQ_{0,(d_i,0)}^{\log}(\PP|H'_i,d)} \arrow[d] \\
{\cQ_{0,2}(\PP,\underline{d})} \arrow[r, "\Delta"] & {\prod_{i=1}^r \cQ_{0,2}(\PP,\underline{d})}                                     
\end{tikzcd}
\end{equation*}
\end{theorem}

\begin{proposition}\label{Prop : rank reduction r}
    With the notations of \Cref{thm:cartesian_log_qmap_snc_r}, we have the following equality of (virtual) classes
    $$[\cQ^{\log}_{(\underline{d},0)}(\PP|H',\underline{d})]^{\vir} = \prod_{i=1}^r(-1)^{d_i + 1} \cdot \ev_1^*H'_i \cap [\cQ_{0,2}(\oplus_{i=1}^r \cO_{\PP}(-H'_i),\underline{d})]^{\vir}.$$
\end{proposition}
\begin{proof}
    As in \Cref{sncloglocalunobstructed}, we have $$[\cQ_{0,2}(\oplus_{i=1}^r \cO_{\PP}(-H'_i),\underline{d})]^{\vir} = \prod_{i=1}^r [\cQ_{0,2}( \cO_{\PP}(-H'_i),\underline{d})]^{\vir}$$
so that, similar to before, it suffices to show that $$[\cQ^{\log}_{(\underline{d},0)}(\PP|H',\underline{d})]^{\vir} = \prod_{i=1}^r [\cQ^{\log}_{0,(d_i,0)}(\PP|H'_i,\underline{d})].$$
As before this follows from \Cref{thm:cartesian_log_qmap_snc_r}.
\end{proof}

A similar argument to \Cref{Lemma : POTpullback} shows

\begin{lemma}
    Let $X=Z \GIT G$ be a GIT quotient and $D = D_1 + \dots + D_r$ a simple normal crossings divisor, satisfying Assumptions \ref{assumptions:absolute_quasimaps} and \ref{assumptions:log_quasimaps}, with $D_1 \cap \dots \cap D_r \neq \emptyset$ and each component very ample. Let $j : X \rightarrow \prod_{i=1}^r \PP^{N_i}$ be the induced morphism with $j^{-1}(H_i \times \prod_{k \neq i} \PP^{N_k}) = D_i$, and assume $H_i$ are coordinate hyperplanes. Then there is an induced morphism $$i : \cQ^{\log}_{0,\alpha}(X|D,\beta) \rightarrow \cQ^{\log}_{0,\alpha}(\PP|H',\underline{d})$$ with $i_{*}\beta = \underline{d}$. Moreover, this morphism admits a perfect obstruction theory and the virtual pullback of the class $[\cQ^{\log}_{0,\alpha}(\PP|H',\underline{d})]$ coincides with the virtual class $[\cQ^{\log}_{0,\alpha}(X|D,\beta)]^{\vir}$ as in ~\cite[Theorem 14]{shafi2024logarithmic}.
\end{lemma}

\begin{proof}[Proof of \Cref{qloglocalcorner}]
        \begin{equation}
        \begin{tikzcd}
{\cQ_{0,(\underline{d},0)}^{\log}(X|D,\beta)} \arrow[d,"p"] \arrow[r,"i"] & {\cQ_{0,(\underline{d},0)}^{\log}(\PP|H',\underline{d})} \arrow[d,"p'"] \\
{\cQ_{0,2}(X,\beta)} \arrow[r,"i'"]                        & {\cQ_{0,2}(\PP,\underline{d})}.                       
\end{tikzcd}
    \end{equation}
Let $\mathscr{L}_i$ (resp. $\mathscr{L}_{D_i}$) denote the universal line bundle on the universal curve $\cQ_{0,2}(\PP,\underline{d})$ (resp. $\cQ_{0,2}(X,\beta)$) corresponding to $\cO_{\PP}(H_i')$ (resp. $\cO_{X}(D_i)$).
\begin{align*}
    &p_*[\cQ_{0,(\underline{d},0)}^{\log}(X|D,\beta)]^{\vir} \\
    &= p_* (i')^{!}[\cQ_{0,(\underline{d},0)}^{\log}(\PP|H',(\underline{d}))]  \\
    &= {i'}^{!}{p'}_{*}[\cQ_{0,(\underline{d},0)}^{\log}(\PP|H',(\underline{d}))] \\
    &= {i'}^{!}\prod_{i=1}^r(-1)^{d_i + 1} \cdot \ev_1^*H'_i \cap [\cQ_{0,2}(\oplus_{i=1}^r \cO_{\PP}(-H'_i),(\underline{d}))]^{\vir} \\
    &= \prod_{i=1}^r(-1)^{d_i + 1} \cdot \ev_1^*{D_i} \cap {i'}^![\cQ_{0,2}(\oplus_{i=1}^r \cO_{\PP}(-H'_i),(\underline{d}))]^{\vir} \\
    &=  \prod_{i=1}^r(-1)^{d_i + 1} \cdot \ev_1^*{D_i} \cap {i'}^!e(\oplus_{i=1}^r\textbf{R}^1\pi_*\mathscr{L}_i^{\vee}) \cap [\cQ_{0,2}(\PP,\underline{d})] \\
    &= \prod_{i=1}^r(-1)^{d_i + 1} \cdot \ev_1^*{D_i} \cap e(\oplus_{i=1}^r\textbf{R}^1\pi_*\mathscr{L}_{D_i}^{\vee}) \cap [\cQ_{0,2}(X,\beta)] \\
    &= \prod_{i=1}^r(-1)^{d_i + 1} \cdot \ev_1^*{D_i} \cap [\cQ_{0,2}(\oplus_{i=1}^r\cO_{X}(-D_i),\beta)]^{\vir}.\qedhere
\end{align*}
\end{proof}
\bibliography{references}
\bibliographystyle{alpha}

\footnotesize

\end{document}